\documentclass[reqno,12pt,letterpaper]{amsart}
\usepackage{amsmath,amssymb,amsthm,graphicx,mathrsfs,url}
\usepackage[usenames,dvipsnames]{xcolor}
\usepackage[colorlinks=true,linkcolor=Red,citecolor=Green]{hyperref}
\usepackage{amsxtra,enumitem}
\usepackage{mathtools}
\mathtoolsset{showonlyrefs,showmanualtags} 
\usepackage{bm}
\usepackage{tikz-cd}

\def\?[#1]{\textbf{[#1]}\marginpar{\Large{\textbf{??}}}}

\def\smallsection#1{\smallskip\noindent\textbf{#1}.}

\newtheorem{theorem}{Theorem}
\newtheorem*{theorem*}{Theorem}
\newtheorem{prop}{Proposition}[section]
\newtheorem{defi}[prop]{Definition}

\newtheorem{lemma}[prop]{Lemma}

\newtheorem{rem}{Remark}[section]

\newtheorem{ex}{Example}[section]
\numberwithin{equation}{section}

\DeclareMathOperator{\Op}{Op}

\DeclareMathOperator{\supp}{supp}

\newcommand*{\dd}{\mathop{}\!\mathrm{d}}

\DeclareMathOperator{\dist}{dist}

\DeclareMathOperator{\Tr}{Tr}

\newcommand{\ip}[2]{\left \langle#1,#2 \right \rangle}
\newcommand{\mat}[1]{\begin{pmatrix} #1 \end{pmatrix}}

\newcommand{\set}[1]{ \left \{ #1 \right \}}
\newcommand{\abs}[1]{\left | #1 \right| }

\newcommand{\N}{\mathbb{N}}
\newcommand{\p}{\partial}
\newcommand{\dbar}{\overline{\p}}
\newcommand{\Z}{\mathbb{Z}}
\newcommand{\T}{\mathbb{T}}

\newcommand{\e}{\varepsilon}

\newcommand{\R}{\mathbb{R}}

\newcommand{\norm}[1]{ \left \| #1 \right \| }

\newcommand{\C}{\mathbb{C}}
\renewcommand{\phi}{\varphi}

\renewcommand{\Re}[1]{{\rm{Re}} \left ( #1\right ) }
\renewcommand{\Im}[1]{{\rm{Im}} \left ( #1 \right ) }

\usepackage[style=alphabetic, backend=biber,doi=false,isbn=false,url=false,eprint=false,giveninits=true]{biblatex}
\title{An Exotic Calculus of Berezin--Toeplitz Operators}
\author{Izak Oltman} 
\address[Izak Oltman]{Department of Mathematics, Northwestern University, 2033 Sheridan Rd, Evanston, IL 60208}
\email{ioltman@northwestern.edu}
\date{\today}

\begin{document}

\setcounter{tocdepth}{1}

\begin{abstract}
We develop a calculus of Berezin-Toeplitz operators quantizing exotic classes of smooth functions on compact K\"ahler manifolds and acting on holomorphic sections of powers of positive line bundles. These functions (classical observables) are exotic in the sense that their derivatives are allowed to grow in ways controlled by local geometry and the power of the line bundle. The properties of this quantization are obtained via careful analysis of the kernels of the operators using Melin and Sj\"ostrand's method of complex stationary phase. We obtain a functional calculus result, a trace formula, and a parametrix construction for this larger class of functions. These results are crucially used in proving a probabilistic Weyl-law for randomly perturbed (standard) Berezin-Toeplitz operators in \cite{Oltman}.
\end{abstract}
\maketitle

\section{Introduction}\label{section:intro}

Let $(X,\omega)$ be a compact, connected, $d$-dimensional K\"{a}hler manifold with a holomorhic line bundle $L$ admitting a positively curved Hermitian metric locally given by $h_\phi = e^{-\phi}$. We assume the Hermitian metric is related to the globally defined symplectic form $\omega$ by the equation $i \p \dbar \phi = \omega$. We call K\"ahler manifolds with these properties quantizable. For further exposition, see Le Floch's presentation in \cite[Chapter 4]{Floch}.

Fixing local trivializations, $\phi$ is locally a strictly plurisubharmonic smooth function. Over each fiber $L_x$ (for $x\in X$), the Hermitian inner-product is written $h_\phi(u,v) = e^{-\phi(x)}u \bar v $ for $u,v\in L_x$. The symplectic form $\omega$ defines a volume form $\mu = \omega^{\land d} /d!$, the Liouville volume form, on $X$. Let $L^N$ denote the $N^{th}$ tensor power of $L$ with Hermitian metric $h^N$. This produces an $L^2$ structure on smooth sections of $L^N$: 
\begin{align}
\ip{u}{v}_{N\phi} := \int _X h_{N\phi} (u,v) \dd \mu , \ \ \ u, v \in C^\infty( X, L^N ) .
\end{align}
We define $L^2 (X,L^N)$ as the completion of $ C^\infty( X, L^N )$ in this  $L^2$ norm. The Bergman projector, denoted by $\Pi_N$, is defined as the orthogonal projection from $L^2 (X,L^N)$ to the subspace of holomorphic sections, denoted by $H^0(X,L^N)$. For any $f\in C^\infty(X;\C)$, we denote by $T_{N,f}$ the Berezin-Toeplitz quantization of $f$ (from now on called Toeplitz operator associated to $f$) which is defined for each $N \in \N$ by
\begin{align}
T_{N,f} (u) := \Pi_N(fu), \ \  u\in H^0 (X,L^N),  \ \ \ T_{N, f } : H^0 (X,L^N) \to H^0 (X,L^N).
\end{align}
Because $H^0(M,L^N)$ is finite dimensional (see \cite[Theorem 4.2.4]{Floch}), $T_{N,f}$ are finite rank operators mapping $H^0 (X,L^N)$ to itself. We now let $f$ depend on $N$, and allow its derivatives to grow in $N$ similarly to the classes $ S_\delta(1) $ in the case of quantization on $ \R^d \times \R^d $ (see \cite[\S 4.4]{Zworski}):
\begin{align}
f \in S_\delta ( 1 )  \ \iff \  \partial^\alpha f = \mathcal O_\alpha ( N^{\delta|\alpha| } ), && \alpha \in \N^{2d}.
\end{align}
In the K\"ahler setting we consider differentiation on a fixed finite set of coordinate patches (see Definition \ref{def:symbol class}). As in the quantization on $ \mathbb R^d \times \mathbb R^d $ we need an additional flexibility of allowing order functions in our symbol classes. This is crucial for our main applications in \cite{Oltman}. The order functions, $m$,  are defined on $ \delta $ scales by demanding that for all $ x, y \in X $, 
\begin{align}
m ( x ) \leq C m ( y )( 1 + N^\delta \dist ( x, y ) )^{M_0} 
\end{align}
for some constants $ C, M_0 > 0 $, and $\delta\in [0,1/2)$. Then we define $ f \in S_\delta ( m ) $ if and only if for each $\alpha \in \N^{2d}$, $\partial^\alpha f = \mathcal O ( N^{\delta |\alpha|} m  )$ on each coordinate patch.

This paper develops a calculus of Toeplitz operators quantizing functions belonging to these more exotic symbol classes. A rough formulation is given as follows.

\noindent
\textbf{Main Result.}
\emph{Suppose that $\delta \in [0,1/2)$, $m_1,m_2$ are $\delta$-order functions on a quantizable K\"ahler manifold $X$, $f \in S_\delta (m_1)$, and $ g\in S_\delta (m_2)$ (see \S \ref{section:symbol class} for definitions). Then
\begin{enumerate}
\item The Schwartz kernels of $T_{N,f}$ and $T_{N,g}$ admit asymptotic expansions. \label{mainresult1}  
\item There exists $h \in S_\delta (m_1m_2)$ such that $T_{N,f} \circ T_{N,g} = T_{N,h} + \mathcal{O} (N^{-\infty})$. \label{mainresult2}
\end{enumerate}}

The precise statements for \ref{mainresult1} and \ref{mainresult2} are given in Theorem \ref{thm.asmptotic} and Theorem \ref{Theorem:composition } respectively. Applications to functional calculus are given in Theorem \ref{thm:functional calculus} and to trace formulas in Theorem  \ref{thm:trace formula}. Coefficients in the expansion of $h$ are given in Appendix \ref{Appendix:1}. More details are also provided at the end of this section.

To put these results in context, we recall that Berezin introduced the concept of Toeplitz operators in \cite{Berezin} to quantize smooth functions (classical observables) on smooth compact symplectic manifolds (classical phase spaces). This generalizes a more straightforward quantization of functions on $ \mathbb T^{2d} := (\mathbb R /2 \pi  \mathbb Z)^{2d} \times \mathbb( \mathbb \R / 2 \pi \mathbb Z)^{2d} $. When $d = 1$, we can think of the coordinate on the first circle, $x\in \T$, as the position variable, and the coordinate on the second circle, $\xi \in \T$, as the momentum variable. Functions $ F = F ( x ) $ and $ G = G ( \xi ) $ are quantized as
\[ \Op_N ( F ) := {\rm{diag}} \left (  F ( 2 \pi j/ N )_{j=0}^{N-1} \right ) , \ \ \ \Op_N ( G ) := \mathcal F_N^*  {\rm{diag}} \left (  G ( 2 \pi j/ N )_{j=0}^{N-1} \right ) \mathcal F_N , \]
where $ \mathcal F_N $ is the unitary discrete Fourier transform on $ \ell^2 ( \mathbb Z_N ) $. This can be generalized to arbitrary functions $ f \in C^\infty ( \mathbb T^2 ) $.
If we consider $ \mathbb T^2 $ as a complex curve and take as $ L$ the theta bundle over it,
one can show that $ \Op_N ( f ) = T_{N,f } + \mathcal O ( N^{-\infty } ) $ (see for instance \cite{Rouby}).
This quantization was used by Vogel in \cite{Vogel} which motivated our main application in \cite{Oltman}. We should also mention that discretizations used in some numerical schemes correspond to Toeplitz quantization on tori (see for instance \cite{Borns}). The analogue of our main result above in the setting of tori can be obtained by using methods already available for the standard quantization in $ \mathbb R^d $ (see \cite{Christiansen}). 

Berezin-Toeplitz quantization of functions on tori has been used as a discrete model of quantum mechanics in both mathematics and physics literature, see for instance \cite{Yonah2022} for a recent application and for pointers in the literature. The physical (rather than purely mathematical) motivation for considering general K\"ahler manifolds with positive line bundles is less clear. We mention however that Anderson in \cite{Andersen2006} and March\'e and Paul in  \cite{Marche2015} studied Toeplitz operators in the context of topological quantum field theory. Deleporte \cite{Deleporte} also used Toeplitz operators to model spin systems in the large spin limit. In \cite{Douglas}, Douglas and Klevstov derived the Bergman projector parametrix for large $N$, which is central to the properties of Toeplitz quantization, using path integrals for particles in a magnetic field.

An essential ingredient needed for this paper is the asymptotic expansion of the kernel of the Bergman projector $\Pi_N$. It was provided by Catlin \cite{Catlin} and Zelditch \cite{Zeldtich} using the Bergman-Szeg\"o kernel parametrix for strictly pseudoconvex domains obtained by Boutet de Monvel and Sj\"ostrand  \cite{Boutet} and extended earlier work by Fefferman \cite{Fefferman}. A direct approach to produce a Bergman kernel expansion for powers of positive line bundles was given by Berman, Berndtsson, and Sj\"ostrand  \cite{BBS}, and another direct approach without relying on the Kuranishi trick was provided by Hitrik and Stone  \cite{Hitrik}.

The main motivation comes from \cite{Oltman} which is devoted to deriving probabilistic Weyl-laws for random perturbations of Toeplitz operators, generalizing the work of Vogel \cite{Vogel} who considered the case of tori. The specific need for developing quantization of exotic classes of functions is the need to study  $\chi ( N^{2 \delta} T_{N, g} )$ where $\chi$ is a smooth cut-off function, $g\in C^\infty (X;\C)$, and $\delta \in [0,1/2)$. This requires a composition formula for Toeplitz operators, which must be established for functions of the form $N^{2\delta} g$. To develop this calculus, the kernel of $N^{2\delta} T_{N,g}$ is explicitly written using the Bergman kernel approximation from \cite{BBS}. The resulting integral is then approximated by Melin and Sj\"ostrand's method of complex stationary phase from \cite{Melin}.

Similar exotic calculi have proven themselves useful in PDE problems from mathematical physics, for instance, long time Egorov theorem, resonance counting, or resolvent estimates. We expect other applications in addition to the motivation coming from \cite{Oltman}.

Composition results for Toeplitz operators of uniformly (in $N$) smooth functions are now standard. They are discussed by Le Floch in \cite{Floch}, Charles in \cite{Charles}, Deleporte in \cite{Deleporte}, Ma and Marinescu in \cite{Marinescu}, and references given there. In Appendix \ref{Appendix:1}, we present a direct computation of the second term in the composition formula, proving the classical-quantum correspondence for Toeplitz operators.

We now state a more precise version of the main results presented in the language of this paper.

\smallsection{Summary of Main Results} For the following four theorems we assume $\delta \in [0,1/2)$ is fixed, $m_1$ and $m_2$ are two $\delta-$order functions on a quantizable $d$-dimensional K\"ahler manifold $X$, $f\in S_\delta (m_1)$, and $g\in S_\delta (m_2)$ (see Definition \ref{def:symbol class}).

\begin{theorem*}[\textbf{Composition Formula}]
There exists $h\in S_\delta (m_1m_2)$ such that
\begin{align}
\norm{T_{N,h} - T_{N,f} \circ T_{N,g}}_{L^2(X,L^N) \to L^2(X,L^N)} = \mathcal{O}(N^{-\infty}). \label{eq:survey 1}
\end{align}
Moreover $h$ {can be locally written as
\begin{equation}\label{eq:survey 3}
\begin{split}
h(x) = f(x)g(x)   - {N^{-1}} \sum _{j,k = 1}^d (\p\dbar \phi(x))^{j,k}&\p_k f(x) \dbar _j g(x) \\
&+ \mathcal{O}(N^{-2(1-2\delta)} m_1(x)m_2(x)) \end{split}
\end{equation}}
where $(\p \dbar \phi(x))^{j,k}$ is such that $\sum _k (\p\dbar \phi(x) )^{j,k} (\p_k \dbar _\ell  \phi(x)) = \delta _{j,\ell}$ for $j,\ell = 1,\dots ,d$.
\end{theorem*}
\begin{theorem*}[\textbf{Trace Formula}]
If $f \sim \sum N^{-(1-2\delta)j}f_j$, then
\begin{align}
\Tr(T_{N,f}) = \left ( \frac{N}{2\pi}\right ) ^d \int_X f_0 (x) \dd \mu(x) + \left ( \int_X m_1(x) \dd \mu(x)\right ) \mathcal{O}(N^{d - (1-2\delta) }). \label{eq:survey 4}
\end{align}
\end{theorem*}

\begin{theorem*}[\textbf{Existence of a Parametrix}]

Suppose $m_1 \ge 1$ and there exists $C> 0$ such that $|f(x)  | > C m_1(x)$ for all $x\in X$. Then there exists $p\in S_\delta (m_1^{-1})$ such that
\begin{align}
T_{N,f} \circ T_{N,p} + \mathcal{O} (N^{-\infty}) = T_{N,p}\circ T_{N,f} + \mathcal{O}(N^{-\infty}) = 1. \label{eq:survey 5}
\end{align}
\end{theorem*}

\begin{theorem*}[\textbf{Functional Calculus}]
Suppose for $x\in X$, $f(x) \in \R_{\ge 0 }$ and there exists $C> 0$ such that $|f(x)|\ge C^{-1} m_1(x) - C$. Then for each $\chi \in C_0^\infty(\R; \C)$, there exists $q\in S_\delta (m_1^{-1})$ such that
\begin{align}
\chi (T_{N,f} ) = T_{N,q} + \mathcal{O}(N^{-\infty})\label{eq:survery 6}
\end{align}
and {$q = \chi (f_0) + O(N^{1-2\delta}m_1^{-1})$ }(i.e. the principal symbol\footnote{If $q\sim \sum_{j=0}^{\infty} N^{-(1-2\delta)}q_j$ (see Definition \ref{def:asymptotic expansion of symbols}), then the principal symbol of $q$ is defined as $q_0$ which is unique modulo $\mathcal O( N^{-(1-2\delta)j}m)$.} of $q$ is $\chi(f_0)$) where $f_0$ is the principal symbol of $f$.
\end{theorem*}

Equation \eqref{eq:survey 1} is proven in Theorem \ref{Theorem:composition }, the first term of the right-hand side of \eqref{eq:survey 3} is proven in Theorem \ref{thm.asmptotic}, the second term of the right-hand side of \eqref{eq:survey 3} is proven in Theorem \ref{thm:second_term}, equation \eqref{eq:survey 4} is proven in Theorem \ref{thm:trace formula}, equation \eqref{eq:survey 5} is proven in Theorem \ref{thm:parametrix construction}, and equation \eqref{eq:survery 6} is proven in Theorem \ref{thm:functional calculus}. 

There are two main difficulties in applying  Melin and Sj\"ostrand's method of complex stationary to our case. First, the amplitude in our integrals is unbounded in $N$, and so any almost analytic extension will also be unbounded in $N$. This growth is carefully controlled by the Gaussian decay of the phase. Second, the critical point of the almost analytically extended phase, which is already an almost analytic extension, must be estimated for proper control of terms in the stationary phase expansion.

\smallsection{Outline of Paper} In Section \ref{section:symbol class} the symbol class $S_{\delta} (m)$ is defined. In Section \ref{section:composition}, the method of complex stationary phase is used to approximate the kernel of $T_{N,f}$ for $f\in S_\delta (m)$. Using the stationary phase expansion, a composition formula is proved for operators in this exotic symbol class. This proof occupies the majority of this paper and becomes fairly technical. However we felt it worth writing the details of how complex stationary phase is applied. For an application of this method in a simpler setting see Appendix \ref{section:AppendixB}. Lastly, in Section \ref{section:application}, this composition formula is used to prove a parametrix construction, a functional calculus, and a trace formula. In Appendix \ref{Appendix:1}, a direct computation of the second term of the composition of Toeplitz operators is presented.

\smallsection{Notation} In this paper we use the following notation. For a function $f$ depending on $N \in \N $ and for a fixed $M \in \R$, we write $f = \mathcal{O} (N^{-M})$ if there exists a constant $C > 0$ such that $|f| \le C N^{-M}$ for all $N$. We write $f = \mathcal{O} (N^{-\infty})$ if $f = \mathcal{O} (N^{-M})$ for all $M \in \N$. We write $f = \mathcal{O} _\alpha (N^{-M})$ if the constant $C$ depends on some parameter $\alpha$. For non-negative functions $f$ and $g$, we write $f\lesssim g$ if there exists a constant $C>0$ such that $f \le C g$. We similarly write $f\lesssim _\alpha g$ if the constant $C$ depends on a parameter $\alpha$. We use the standard multi-index notation with the following twist: if $\alpha \in \N^{2d}$ and $f\in C^\infty(\C^d ; \C)$, then $\p^\alpha_{x,\bar x}f (x):=\left ( \prod _{j,k=1}^{d}\p^{\alpha_j}_{x_j} \p^{\alpha_{k+d}}_{\bar x_k} f \right ) (x)$, where $\p_x : = \frac{1}{2} (\p_{\Re{x}} - i \p_{\Im{x}})$ and $\p_{\bar x} : = \frac{1}{2} (\p_{\Re{x}} + i \p _{\Im{x}})$ are the holomorphic and anti-holomorphic derivative operators respectively.

\section{A new symbol class}\label{section:symbol class}

For the remainder of this paper, we fix a finite atlas of neighborhoods $(U_i , \rho_i ) _{i \in \mathcal{I}}$ for a fixed K\"{a}hler manifold $X $ as described in the introduction (a compact, quantizable, $d$-dimensional K\"ahler manifold).
Without loss of generality, we assume $\rho_i (U_i)  = \C^d$.

\begin{defi}[\textbf{$\bm{\delta}$-order Function on $\bm{X}$}]\label{def:order function}
For $\delta \in [0,1/2)$, a $\delta$-order function on $X$ is a function $m\in C^\infty (X;\R_{> 0})$, depending on $N$, such that there exist $C,M_0 >0$ so that for all $x,y\in X$ 
\begin{align}
m(x) \le C m(y) \left (1 + N^\delta \dist (x,y) \right )^{M_0} \label{order function}
\end{align}
where $\dist(x,y)$ is the distance between $x$ and $y$ with respect to the Riemannian metric on $X$ induced by the symplectic form $\omega$. 
\end{defi}

This paper will also use $\delta$-order functions on $\R^d$ and $\C^d$, which are defined below.

\begin{defi}[\textbf{$\bm{\delta}$-order Function on $\R^d$ and $\C^d$}]
For $\delta \in [0,1/2)$, a $\delta$-order function on $\R^d$ (or $\C^d$) is a function $m\in C^\infty (\R^d;\R_{> 0})$ (or $C^\infty(\C^d ; \R_{>0})$, depending on $N$, such that there exist $C,M_0 >0$ so that for all $x,y\in \R^d$ (or $\C^d$) 
\begin{align}
m(x) \le C m(y) \left (1 + N^\delta |x-y| \right )^{M_0} 
\end{align}
\end{defi}

\begin{ex}
If $f \in C^\infty (X ; \R_{\ge 0} )$ and $\delta \in [0,1/2)$, then $ m = N^{2\delta} f + 1$ is a $\delta$-order function on $X$.
\end{ex}
\begin{proof}
First let $x,y\in U_i$ for some $i \in \mathcal{I}$ and define $m_i = m \circ \rho_i^{-1}$ and $f_i = f \circ \rho_i^{-1}$. Using that $m_i \ge 1$, for $\alpha \in \N^{2d}$ with $|\alpha| = 1$,
\begin{align}
(\p^\alpha_{x,\bar x} m_i )(x) = N^{2\delta } (\p_{x,\bar x }^\alpha f_i ) (x) \lesssim N^{2\delta} \sqrt{f_i (x)} \le 	N^{\delta} \sqrt{m_i(x)}.
\end{align} 
Here we use the fact that if $g \in C^\infty (\C^d; \R_{\ge 0})$ with bounded derivatives and $|\alpha |= 1$ then $|	\p^\alpha_{x,\bar x } g(x) | \lesssim \sqrt{g(x)} $ (see for instance \cite[Lemma 4.31]{Zworski}).

If $\alpha \in \N^{2d}$ with $|\alpha | = 2$, then because $f$ is bounded, $(\p^\alpha _{x,\bar x } m_i) (x) \lesssim N^{2\delta}$. So by Taylor expansion, there exists a $C> 0 $ such that
\begin{align}
m_i (x) &\le m_i (y) + C  ( \sqrt {m_i(y)} |x- y| N^{ \delta} +  |x - y|^2 N^{2\delta}) \label{eq:205} \\
&\lesssim (1 + |x-y|^2 N ^{2\delta} )m_i (y). \label{eq:206}
\end{align}
To see this last inequality, let $a = \sqrt{m_i(y)}$ and $b = |x - y|N^{\delta} $, then using that $a \ge 1$ and $b \ge 0$, the right-hand side of \eqref{eq:205} is:
\begin{align}
a^2 +C( ab + b^2) \lesssim (a^2 + b^2)\lesssim (a^2 + a^2 b^2)
\end{align}
which is the right-hand side of \eqref{eq:206}. As $X$ is compact, there exists a $C>0$ such that
\begin{align}
\frac{1}{C} |\rho_i(x) - \rho_i(y) | \le \dist(x,y) \le C |\rho_i (x) - \rho_i (y) |
\end{align}
for all $x,y\in U_i$. Therefore $m$ satisfies \eqref{order function} on the patch $U_i$ with $M_0 = 2$.

For the global statement, pick $x,y\in X$, and consider the minimum number of charts that cover a geodesic from $x$ to $y$ having length $\dist(x,y)$. Next, label these charts $U_0, \dots, U_M$ where $x\in U_0$ and $y\in U_M$. For each $U_i$ ($i \neq M$), select some $z_i \in U_i \cap U_{i+1} $. Then by the above, we have that
\begin{align}
m(x) \lesssim m(z_0) \left ( 1 + \dist(x,z_0) N^\delta \right )^2 \lesssim \cdots \lesssim m(y) \prod _{i=-1}^{M-1}\left (1 + \dist(z_i , z_{i+1} ) N^\delta \right )^2 
\end{align}
(where $z_{-1} := x$ and $z_M := y$). Then by the selection of the charts, $\dist(z_i,z_{i+1}) \le \dist(x,y)$, so that:
\begin{align}
m(x) \lesssim m(y) \left ( 1+ \dist(x,y) N^\delta\right )^{2M}.
\end{align}
Therefore we have \eqref{order function} for $M_0 = 2|\mathcal{I}|$ (where $|\mathcal{I}|$ is the number of charts).
\end{proof}

\begin{defi}[$\bm{S_\delta (m)}$]\label{def:symbol class}
Let $m$ be a $\delta$-order function on $X$ (with $\delta \in [0,1/2)$ fixed). Define $S_\delta (m)$ as the set of all smooth functions $f$ on $X$, which are allowed to depend on $N$, such that for all $\alpha \in \N^{2d}$, there exists $C_\alpha > 0$ such that
\begin{align}
|\p^\alpha_{x,\bar x } (f \circ \rho_i ^{-1} (x) ) | \le C_\alpha N^{\delta |\alpha|} m\circ \rho_i^{-1} (x) \label{eq:253}
\end{align}
for each $i\in \mathcal{I}$ and $x\in \rho_i (U_i)$.
\end{defi}

Similarly, by replacing $X$ by $\R^d$ or $\C^d$, we can define $S_\delta (m)$ for functions on $\R^d$ or $\C^d$.

\begin{defi}[\textbf{Asymptotic Expansion of Symbols}]\label{def:asymptotic expansion of symbols}
For $\delta \in [0,1/2)$, $m$ a $\delta$-order function on a quantizable K\"ahler manifold $X$, functions $f_j \in S_\delta(m)$ {(which may depend on $N$)} for $j\in \Z_{\ge 0}$ and $f \in S_\delta (m)$, we will write $f\sim \sum _0^\infty N^{-(1 - 2\delta )j} f_j$ if for all $\alpha \in \N^{2d}$, $M \in \N$, and $i \in \mathcal{I}$
\begin{align}\begin{split}
\p_{x,\bar x }^\alpha \Big ( f\circ \rho_i^{-1} (x) - \sum _{j = 0}^{M-1} N^{-(1-2\delta )j}& (f_j \circ \rho_i ^{-1} (x)) \Big ) \\
 &=\mathcal{O}_{\alpha,M}(N^{-M(1-2\delta)+ |\alpha |\delta} m\circ \rho_i ^{-1} (x)).  \end{split}\label{eq:300}
\end{align}
\end{defi}

By Borel's Theorem (see \cite[Theorem 4.15]{Zworski} for instance), given any $f_j \in S_\delta (m)$ we can always construct such an asymptotic sum.

\begin{prop}[\textbf{Borel's Theorem for $\bm{S_\delta (m)}$}]\label{prop:Borel}
Fixing $\delta \in [0,1/2)$, $m$ a $\delta$-order function on a quantizable K\"ahler manifold $X$, and $f_j \in S_\delta (m)$ for $j\in \Z_{\ge 0}$, then there exists $f\in S_\delta (m)$ such that \eqref{eq:300} holds and $f$ is unique modulo $\mathcal{O}(N^{-\infty})$ error.
\end{prop}
\begin{proof}
On each coordinate patch, $U_i$, define
\begin{align}
a_j (x) : = N^{-(1 - 2\delta)j } f_j ( \rho_i^{-1} (xN^{-\delta})) && \text{and} && \tilde m_i (x) : = m( \rho_i^{-1} (xN^{-\delta})).
\end{align}
In this case $a_j \in S_0 (\tilde m_i )$ for each $j$, and so by \cite[Theorem 4.15]{Zworski}, there exists $a\in S_0 (\tilde m_i)$ such that $a\sim \sum _0^\infty N^{-j(1-2\delta ) } a_j $ with uniqueness modulo $\mathcal{O}(N^{-\infty})$ error. On $U_i$, we let $f(x)= a(\rho(x)N^{\delta})$. We can glue each patch together by uniqueness to get a globally defined $f$.
\end{proof}

\begin{defi}[\textbf{Principal Symbol}]
Fixing $\delta \in [0,1/2)$, $m$ a $\delta$-order function on a quantizable K\"ahler manifold $X$, and $f\in S_\delta (m)$ given as an asymptotic sum $f\sim \sum_0^\infty N^{-(1-2\delta )j } f_j$ for $f_j \in S_\delta (m)$, then the principal symbol of $f$ is defined as the equivalence class of functions $f_0$ with the equivalence $f_0' \sim f_0$ if $f_0'-f_0 \in N^{-(1-2\delta)} S_\delta (m)$.
\end{defi}

We remark that if $f\in S_\delta (m)$ does not admit an asymptotic expansion, its principal symbol can still be defined as the equivalence class of functions $f_0$ such that $f - f_0 \in N^{-(1-2\delta) }S_\delta (m)$.

\begin{ex}	
Given a compact K\"ahler manifold $X$, smooth functions $f_j \in S_0(1)$ for $j\in \Z_{\ge 0}$, $f_0 \ge 0$, $f \sim \sum_0^\infty N^{-j} f_j $, $\delta \in [0,1/2)$ and $m = f_0 N^{2\delta } + 1$, then $f N ^{2\delta }\in S_\delta (m)$.	
\end{ex}	
\begin{proof}	
For each $i$, let $g _i = N^{2\delta} f\circ \rho_i^{-1}$ and $m_i = m \circ \rho_i^{-1}$. First, for $x\in U_i$	
\begin{align}	
|g_i (x)| \lesssim N^{2\delta} f_0 \circ \rho_i^{-1} (x) \lesssim m_i(x).	
\end{align}	
If $\alpha \in \N^{2d}$ with $|\alpha | = 1$, then:	
\begin{align}	
|\p^\alpha_{x,\bar x } g _i (x) | \lesssim N^{2\delta } \p^\alpha_{x,\bar x } \left (f_0 \circ \rho_i^{-1} \right ) (x) \lesssim N^{2\delta} \sqrt{f_0 \circ \rho_i^{-1} (x) } \lesssim N^\delta m_i (x),	
\end{align}	
using that $0 \le f_0< m N^{-2\delta}$ and $m \ge 1$. Then, because each $f_i$ is bounded, for all $\alpha \in\N^{2d}$ with $|\alpha | \ge 2$	
\begin{align}	
|\p^\alpha_{x,\bar x } g_i (x) | \lesssim N^{2 \delta } \p^\alpha_{x,\bar x } \left (f_0 \circ \rho_i^{-1}\right)(x ) \lesssim N^{2\delta } \le m(x) N^{2\delta} \le m (x)N^{ \delta |\alpha|}.	
\end{align}	
\end{proof}

\subsection{Almost Analytic Extension}

When applying the method of complex stationary phase, almost analytic extensions of smooth functions are constructed. We will briefly review results about almost analytic extensions, as well as prove estimates for almost analytic extensions of functions in $S_\delta (\R^d)$. Similar results about almost analytic extensions of functions with growth in $N$ are proven in \cite[Proposition 1.16]{Melin}.

We recall the notation $\p_z = \frac{1}{2} (\p_{_\Re z } - i \p_{\Im z})$ and $\dbar_z = \frac{1}{2} (\p_{\Re z} + i \p _{\Im z})$ to denote holomorphic and anti-holomorphic differentiation. Recall $f\in C^\infty(\C)$ is holomorphic in an open set $U\subset \C$ if and only if $\dbar _z f (z) = 0$ for all $z\in U$. To apply the method of complex stationary phase, we would like to take a smooth compactly supported function $f$ on $\R^d$ and extend it to a holomorphic function $\tilde f$ on $\C^d$ and apply the Cauchy integral formula. Requiring holomorphy of $\tilde f$ is impossible by Liouville's theorem. However, if we relax the condition of holomorphy to $\dbar \tilde f$ vanishing up to infinite order as we approach the real axis, we can apply a variant of the Cauchy integral formula.

For a smooth function $f \in C^\infty(\C^d)$, we write $f = \mathcal{O} (\abs{\Im{z}}^\infty)$ to mean that for any $M \in \N$ and compact set $K\subset \C$, there exists $C = C(M,K) > 0$ such that
\begin{align}
|f(z) | \le C \abs{\Im z}^M
\end{align}
for all $z\in K$.

\begin{prop}[\textbf{Almost Analytic Extension of $\bm{C_0^\infty(\R^d)}$ Functions}] \label{thm:aa e}
If $f \in C_0^\infty (\R^d)$, then there exists $\tilde f \in C_0^\infty(\C^d)$ such that for $\alpha,\beta \in \N^d$, $|\beta |\ge  1$ 
\begin{enumerate}
\item $\tilde f |_{\R^d} = f$,
\item $\dbar_z \tilde f (z) = \mathcal{O} (\abs{ \Im z } ^\infty) $,
\item $\p^\alpha_z \dbar^\beta_{ z } \tilde f(z) = \mathcal{O}_{\alpha,\beta} ( \abs{\Im z }^\infty)$,
\item $\p^\alpha _z \tilde f(z) = \mathcal{O}_\alpha (1)$.
\end{enumerate}
Moreover, given any neighborhood containing the support of $f$ in $\C^d$, such a $\tilde f$ can be constructed to be supported in this neighborhood.
\end{prop}

\begin{rem}
While included for clarity, we note that conditions $(2)$ and $(3)$ are equivalent in Proposition \ref{thm:aa e} as proven in \cite[Chapter 10, Lemma 2.2]{Treves}.
\end{rem}

A construction in one dimension of such an extension (which is easily generalized to higher dimensions) is 
\begin{align}
\tilde f(x+iy) =\frac{\psi(x)}{2\pi}\int _\R e ^{i(x + iy)\xi } \hat f(\xi) \chi (\xi y) \dd\xi 
\end{align}
where $\psi,\chi \in C_0^\infty(\R)$, with $\chi \equiv 1$ near $0$ and $\psi \equiv 1$ on the support of $f$. See \cite[Chapter 10.2]{Treves} for further discussion.

Almost analytic extensions are not unique. However, if $\tilde f$ and $\tilde g$ are two almost analytic extensions of $f$ on $\R^d$, then by Taylor expansion, $\p^\alpha_{z,\bar z } (\tilde f(z) - \tilde g(z)) = \mathcal{O}_\alpha (\abs{\Im z }^\infty)$ for any $\alpha \in \N^{2d}$.

Furthermore, smooth functions can be extended off any totally real subspace, that is a subspace $V \subset \C^d$ such that $i V \cap V = \set{0}$. In this way if $f \in C^\infty_0(V)$, then there exists $\tilde f \in C^\infty (\C^d)$ such that $\dbar_z \tilde f (z) = \mathcal{O} (\abs{ \dist(z,V) }^\infty)$. While any holomorphic function is determined by its restriction to a maximally totally real subspace, the same is true for almost analytic extensions modulo $\mathcal{O} (\abs{ \dist (z,V) }^\infty)$ error.

\begin{prop}[\textbf{Almost Analytic Extensions of $\bm{S_\delta (m)}$ Functions on $\bm{\R^d}$}]\label{thm.extension of symbols}
Given a $\delta$-order function $m$ on $\R^d$ (with $\delta \in [0,1/2)$), and $f\in S_\delta (m)\cap C_0^\infty (\R^d)$, then there exists $C>0$ and an almost analytic extension $\tilde f \in C^\infty_0(\C^d)$ such that for $\alpha,\beta \in \N^d$, $|\beta| \ge 1$,
\begin{enumerate}
\item$\tilde f |_{\R^d } = f$,\label{item:1}
\item $\dbar_z \tilde f (z) = m (\Re z ) N^{\delta } \mathcal{O} ( \abs{\Im z N^{\delta} }^\infty)$, \label{item:2}
\item $\p_z^\alpha \dbar_{z}^\beta  \tilde f(z) = \mathcal{O}_{\alpha,\beta} (|\Im z N^{\delta} |^\infty ) N^{\delta (1 + |\alpha | + |\beta|)} m(\Re{z}) $,\label{item:3}
\item $\p_z^\alpha \tilde f(z) = \mathcal{O}_\alpha ( N^{\delta |\alpha|}) m (\Re{z})$,\label{item:4}
\item $\supp \tilde f \subset \set{\abs{\Im z } < C N^{-\delta} }$.\label{item:5}
\end{enumerate}

\end{prop}
\begin{proof}
Let $\psi \in C_0^\infty(\R^d)$ be such that $\psi_n(x): = \psi(x-n)$ for $n\in \Z^d$ is a smooth partition of unity. Next let $g(x) = f(N^{-\delta }x) $ and $\tilde m (x)= m(N^{-\delta} x )$, so that for each $\alpha \in \N^{2d}$, $|\p^\alpha_{x,\bar x} g(x)|\lesssim_\alpha N^{-\delta |\alpha | } N^{\delta |\alpha |} m (N^{-\delta} x) \lesssim_\alpha \tilde m (x) $.

For each $n\in \Z^d$, let $g _n (x) = g(x) \psi_n (x)$.
{Let $c_d>0$ be such that $\supp g_n \subseteq  \supp \psi_n \subset \set{|x - n| < c_d}$.} 
{Define $\tilde m_n := \max_{|x - n| \le c_d} \tilde m (x)$.}
Note that $\tilde m(x) /\tilde m (y) \le C (1 + |x-y| )^{M_0}$ so that ${ C^{-1} (1+ c_d)^{-M_0} \tilde m _n \le  m(x) \le C (1+c_d)^{M_0} \tilde m_n}$ for all $x \in \supp g_n$. Therefore for all $\alpha \in \N^d$, there exists $C_\alpha > 0$ such that $|\p^\alpha g_n (x) | \le C_\alpha \tilde m_n$. Then, by Proposition \ref{thm:aa e}, an almost analytic extension of $g_n$, $\tilde g_n$, can be constructed to satisfy the following.
\begin{itemize}
\item $\tilde g_n | _{\R^d} = g_n$.
\item $\dbar_z \tilde g_n (z) = \tilde m _n \mathcal{O} (\abs{\Im z } ^\infty)$ (with constants independent of $n$).
\item For each $\alpha,\beta \in \N^d$, $|\beta | \ge  1$: $\p^\alpha_z \dbar _z^\beta  \tilde g_n (z) = \tilde m_n \mathcal{O}_{\alpha,\beta} (\abs{\Im z }^\infty) $.
\item For all $\alpha \in \N^d$: $|\p_{z} ^\alpha \tilde g_n (z)| = \tilde m_n \mathcal{O} _\alpha (1)$,
\item $\supp \tilde g_n$ is contained within a complex neighborhood of the real ball of radius {$c_d$} centered around $n$.
\end{itemize}

\noindent Since $\sum g_n = g$, it follows that $\sum g_n (xN^{\delta} ) =f(x)$. Therefore the natural choice of an extension of $f$ is $\tilde f(z): = \sum \tilde g_n (zN^{\delta})$.

From this, \eqref{item:1} follows immediately. To see \eqref{item:2}, let $\tilde g (z) : = \tilde f (N^{-\delta} z)$, then for all $M\in \N$:
\begin{align} 
|\dbar \tilde g(z) | \le \sum _{n : z \in \supp \tilde g_n} |\dbar _{z} \tilde g_n (z) | \lesssim_M \abs{\Im z }^M \sum _{n : z \in \supp \tilde g_n} \tilde m_n \lesssim_M \abs{\Im z }^M \tilde m(\Re z ) .
\end{align}
Changing variables, $z \to zN^{\delta}$, we get:
\begin{align}
|\dbar f(z)| \lesssim_M \abs{\Im z }^M N^{\delta M -1} m (\Re z ) .
\end{align}
By the same change of variables, \eqref{item:3}, \eqref{item:4}, and \eqref{item:5} follow similarly.
\end{proof}

\section{Composition of Toeplitz operators with symbols in \texorpdfstring{$S_\delta(m)$}{Sδ(m)}} \label{section:composition}

This section estimates the composition of two symbols in $S_\delta (m_1)$ and $S_\delta(m_2)$ where $m_1$ and $m_2$ are two $\delta$-order functions on $X$ with $\delta\in[0,1/2)$ fixed. That is, if $ f \in S_\delta (m_1)$ and $g \in S_\delta (m_2)$, then this section constructs $h \in S_\delta (m_1 m_2 ) $, such that $T_{N,f} \circ T_{N,g} = T_{N, h} + \mathcal{O} (N^{-\infty})$. Here the big-O is in terms of the norm $L^2 (X,L^N) \to L^2 (X,L^N)$.

This proof will be broken into several steps which rely on the method of complex stationary phase. Before this, we explicitly write out the Schwartz kernel of Toeplitz operators. A brief digression into working with smooth sections is helpful.

\subsection{Holomorphic Sections}\label{section:holomorphic sections}

On the holomorphic Hermitian line bundle $L$ over $X$, let $\theta_j$ be trivializations on the open sets $U_j$ with transition functions $g_{jk}$ defined by: $\theta_j \circ \theta_k ^{-1} (x,t) = (x,g_{jk}(x) t)$ for $x\in U_j$ and $t\in \C$. Sections on $L$ can be locally written $s(x) = s_j (x) e_j (x)$ where $e_j (x) := \theta_j ^{-1} (x,1)$ and $s_j$ are complex valued functions. A global section $s$, given by $s_j$'s, must obey the transition rule $s_k (x)= g_{kj} (x) s_j (x)$. Recall that lengths of elements of $L$ coincide with the K\"{a}hler potential $\phi$, that is $\norm{e_j (x)} = e^{-\phi_j (x)}$. 

Using the volume form $\mu : = \omega^{\land d} / d!$, an $L^2$ inner-product on sections of $L$ is explicitly written
\begin{align}
\ip{u}{v} = \sum_j \int _{U_j} \chi_j (x) u_j (x) \bar v_j (x) e ^{- \phi_j (x) } \dd\mu (x),
\end{align}
where $\chi_j$ is a partition of unity subordinate to $U_j$, and $u,v$ are smooth sections such that $u = \sum u_j e_j$ and $v = \sum v_j e_j$. In this way, smooth sections can locally be described by smooth functions. Throughout this paper, as most of the analysis is local, sections are treated as smooth functions and $e_j$ are not written.

Sections of the $N^{th}$ tensor power of $L$, denoted by $L^N$, are locally written $s(x) = s_j (x) e_j (x)^{\otimes N}$ where $s_j$ obey the transition rule $s_k (x) = g_{kj}^N (x) s_j (x)$. A Hermitian metric on $L^N$ can be constructed by raising the original metric to the $N^{th}$ power. In this way, there is a natural inner product on sections of $L^N$ coming from the original metric. Given sections $u$ and $v$ on $L^N$, define
\begin{align}
\ip{u}{v} = \sum _j \int _{U_j} \chi _j (x) u_j (x) \bar v_j (x) e^{-N \phi_j (x) } \dd\mu (x).
\end{align}
In the following section, we will write operators on sections as integral kernels. The Bergman projector $\Pi_N$ is a bounded map from $L^2 (X, L^N) \to L^2 (X,L^N)$, and by the Schwartz kernel Theorem (see \cite[Proposition 6.3.1]{Floch}), has Schwartz kernel $\Pi_N \in L^2 (X \times \bar X , L^N \boxtimes \overline{L^N})$. Here $\bar X$ is the manifold $X$ with symplectic form $-\omega$ and complex structure opposite to the complex structure on $X$. $\Pi_N$ is a smooth section of $L^N \boxtimes \overline{L^N}$ which is holomorphic in the first argument and anti-holomorphic in the second argument, so we write $\Pi_N \in H^0 (X \times \overline X, L^N \boxtimes \overline{L^N})$. Similarly, Schwartz kernels of Toeplitz operators also live in $H^0 (X \times \overline X, L^N \boxtimes \overline{L^N})$. Just as smooth sections can locally be described by smooth functions, smooth sections of $L^N \boxtimes \overline{L^N}$ can be described by smooth functions of two variables.  A holomorphic section of $L^N \boxtimes \overline{L^N}$ can be locally defined by a function of two variables which is holomorphic in the first component, and anti-holomorphic in the second component. For this reason, we will locally write $\Pi_N$ as $\Pi_N(x,\bar y)$, so that the function $\Pi_N(x,y)$ is holomorphic in both variables. In the remainder of this paper, this convention will be used for both holomorphic functions, and almost-holomorphic functions.

\subsection{Bergman Kernels} \label{subsection:Bergman kernels}
For a smooth function $f$ on $X$, the associated Toeplitz operator has kernel
\begin{align}
T_{N,f} (x,\bar y) = \int_X \Pi_N(x,\bar w) f(w) \Pi_N (w,\bar y) e^{-N\phi(w)} \dd \mu(w).
\end{align}
This paper does not include the $L^2$ weight in the kernel, that is if $u$ is a smooth section on $L^N$, then
\begin{align}
T_{N,f} [u](x) := \int _X T_{N,f} (x,\bar y) u(y) e^{-N \phi(y) } \dd\mu(y).
 \end{align}
In \cite{BBS}, Berman, Berndtsson, and Sj\"{o}strand provided a direct proof to approximate the Bergman kernel near the diagonal (for each $N_0 \in \N$) by:
\begin{align}
\Pi^{N_0} _N (x,\bar y) = \left ( \frac{N}{ 2 \pi}\right ) ^d e^{N \psi (x,\bar y) } \left (1 + \sum_1^{N_0} b_i (x,\bar y) \right ) \label{b_is}
\end{align}
where $b_m \in C^\infty (X \times X ; \C)$ are locally almost analytic off $\set{(x,\bar x)}$, and $\psi$ is an almost analytic extension of $\phi$ such that $\psi(x,\bar x) = \phi(x)$. By \cite{BBS},
\begin{align}
\Pi_N (x,\bar y) = \Pi_N^{N_0} (x,\bar y) + \mathcal{O} (N^{d - N_0 - 1} ) e^{ \frac{N}{2} ( \phi (x) +\phi (y) )}. \label{eq:399}
\end{align}

Conversely, away from the diagonal we have the following Lemma proven in \cite{Ma}.

\begin{lemma}[\textbf{Off Diagonal Decay of Bergman Kernel}]\label{lemma:offdiagonal}
For $x,y\in X$, there exists $C,c > 0 $ such that:
\begin{align}
\norm{\Pi_N (x, \bar y)}_{h_N}\le C N^d e^{-c \sqrt{N} \dist(x,y)}
\end{align}
where $\dist(\cdot ,\cdot ) $ is the Riemannian distance on $X$ and locally $\norm{\Pi_N (x,\bar y)}_{h_{N }}$ is
\begin{align}
e^{-\frac{N}{2} (\phi(x) + \phi(y))} |\Pi_N (x,\bar y) |. \label{eq:384}
\end{align}
\end{lemma}
It should be noted that in \cite{Christ}, Christ proved a stronger decay estimate of $\Pi_N$ away from the diagonal, but we want to avoid fine analysis at this early stage.

By Taylor expanding $\psi (x,\bar y)$ near the diagonal, it follows that there exists $C> 0$ such that
\begin{align}
\Re {\psi (x,\bar y)} \le -C |x - y|^2 + \frac{1}{2} (\phi (x) +\phi(y)) \label{eq.gaussian behaviour}
\end{align}
(see for instance \cite[proposition 2.1]{Hitrik}). This, along with Lemma \ref{lemma:offdiagonal}, provides the following global bound.
\begin{lemma}[\textbf{Global Bergman Kernel Bound}]\label{lemma.global bergman kernel bound}
There exists $\e,C,c > 0 $ such that:
\begin{align}
\norm{\Pi_N (x,\bar y)}_{h_{N }} \le \begin{cases}
C N^d e^{-C N |x -y|^2 } + \mathcal{O} (N^{-\infty}) & \text{if} \ \dist(x,y) < \e, \\
C N^d e ^{- c \sqrt{N} \e } & \text{if} \ \dist(x,y) > \e .
\end{cases} 
\end{align}
Where, locally $\norm{\Pi_N (x,y)}_{h_{N }}$ is given by \eqref{eq:384} and $|x- y|$ is the distance in coordinates of $x$ and $y$ ($\e > 0 $ is chosen sufficiently small such that $x$ and $y$ are in the same chart).
\end{lemma}

\begin{rem}[\textbf{Factor of Two}]
By \cite[Proposition 3.5.6]{Floch}, a quantizable K\"ahler manifold will have a unique Hermitian metric, up to multiplication by a constant. The choice of such a constant varies in the literature. When discussing Bergman kernel asymptotics, the convention is to take $(i/2)\p \dbar \phi = \omega$, while when discussing Toeplitz quantization, the convention is to take $i \p \dbar \phi = \omega$. The latter is more natural in the semiclassical setting (and followed in this paper). This is the reason why the Bergman kernel used in this paper (and all literature discussing Toeplitz quantization) is $2^{-d}$ times the Bergman kernel in papers discussing the Bergman kernel asymptotics (for instance in \cite{BBS}).
\end{rem}

\subsection{An Asymptotic Expansion of the Kernel of a Toeplitz Operator}

In this subsection, we apply Melin and Sj\"ostrand's method of complex stationary phase to obtain an asymptotic expansion of the kernel of Toeplitz operators for functions in $S_\delta (m)$.
 
\begin{theorem}[\textbf{Asymptotic Expansion of Symbols in $\bm{S_\delta (m)}$}]\label{thm.asmptotic}
Suppose $\e> 0 $ is small enough such that if $\dist (x,y) < \e$, then $x$ and $y$ are contained in the same chart. Let $ \Delta_\e = \set{(x,y) \in M \times M : \dist (x,y ) < \e}$. Fixing $\delta \in [0,1/2)$, suppose $m$ is a $\delta$-order function on $X$ with constant $M_0$ in \eqref{order function}. Then if $f\in S_\delta(m)$, there exist $f_j \in C^\infty ( \Delta_\e ; \C)$ ($j\in \Z_{\ge 0}$) such that for all $J\in \N$, in local coordinates $T_{N,f}(x,\bar y)$ is
\begin{align}
\left ( \frac{N}{2\pi}\right ) ^d e^{N\psi (x,\bar y) } \left ( \sum _{j=0}^{J-1} N^{-j} f_j (x,\bar y ) + N^{-J}R_{J} (x,\bar y) \right ) + e^{\frac{N}{2} (\phi(x) + \phi(y) )} \mathcal{O}(N^{-\infty} ) \label{eq.asymptotic}
\end{align}
where:
\begin{align}
R_J (x,\bar y) &\in N^{2\delta J } ( S_\delta (m(x)) \cap S_\delta (m(y)) , \\
f_j(x,\bar y ) &\in N^{2\delta j} \left (  S_\delta (m(x)) \cap S_\delta (m(y)) \right ),\\
\supp f_j (x,\bar y) &\subset \set{\dist(x,y) < C N^{-\delta}},\\
f_0 (x,\bar x) &= f(x),
\end{align}
for some $C > 0$. Moreover, in local coordinates, $f_j (x, y)$ are almost analytic off the totally real submanifold $\set{(x,y ) \in \C^d \times \C^d : y= \bar x}$, and if $f_j'$ is another almost analytic extension agreeing with $f_j$ on the diagonal, then the difference of the two kernels is $\exp(\frac{N}{2} (\phi(x)+ \phi(y) ) \mathcal{O} (N^{-\infty})$. 
\end{theorem} 

\begin{rem}
An alternate asymptotic expansion can be written by bounding $R_J$ in \eqref{eq.asymptotic} and absorbing the $\exp((N/2)(\phi(x) + \phi(y)) \mathcal{O}(N^{-\infty})$ term. That is (with the same quantifiers as in Theorem \ref{thm.asmptotic}) $T_{N,f}(x,\bar y)$ can be written
\begin{align}\begin{split}
\left ( \frac{N}{2\pi} \right )^d e^{N \psi(x,\bar y)} \sum _{j=0}^{J-1} N^{-j}& f_j (x,\bar y) \\
&+ e^{\frac{N}{2} (\phi(x) + \phi(y) ) } \mathcal{O}(N^{ d - J(1 - 2\delta)}\min(m(x),m(y)) ). \end{split} \label{eq:asymp_better}
\end{align}
\end{rem}

The proof follows the method of complex stationary phase, developed by Melin and Sj\"{o}strand in \cite{Melin}, and presented in \cite{Treves} by Tr\`eves. The difficulty is that the amplitude is no longer bounded in $N$, but lives in $S_\delta (m)$ and so any almost analytic extension is slightly weaker than in \cite{Treves}. Careful analysis is required to ensure the stationary phase still provides appropriate remainders.

To avoid reproving the method of complex stationary phase, we use the same notation for variables in \cite{Treves}. Unfortunately, there is no ideal uniform choice of variables. We will have the same variable used for different objects in separate parts of the proof. We begin with $(x,y)$ to denote the argument of the Schwartz kernel. For each $(x,y)$, we get an integral over $w\in \C^d$. We rewrite $w$ in real coordinates, $p$, and replace $(x,y)$ by $t$, to use the same notation as in \cite[Chapter 10]{Treves}. After going through complex stationary phase, we replace $t$ by $(x,y)$. This notational choice is summarized in the following table for the reader's convenience.

\begin{tabular}{|c|c|c|c|c|}
\hline 
variable name & space & first reference & step(s) used & comment \\ 
\hline 
$x,y$ & $\C^d$  & \eqref{eq.pg6} & 1,5-9 & argument of kernel \\ 
\hline 
$w$ & $\C^d$ & \eqref{eq.pg6} & 1,5 & integrated variable \\ 
\hline 
$p$ & $\R^{2d}$ & \eqref{eq:528} & 1 & realifies $w$ \\ 
\hline 
$p$ & $\C^{2d}$ & step 2 & 2-4 & complexifies previous $p$ \\ 
\hline 
$t$ & $\R^{4d}$ & \eqref{eq:529} & 1-5 & realifies $(x,\bar y)$ \\ 
\hline 
${p_{\rm{crit}}}(t) ={p_{\rm{crit}}}(x,\bar y)$ & $\C^{2d}$ & step 2 & 2,4-8 & critical point of $\tilde \Psi$ \\ 
\hline 
$p(x,\bar y ,z) = p (z_0)$ & $\C^{2d}$ & step 2 & 2,3,5,7 & point on new contour \\ 
\hline 
\end{tabular} 

\begin{proof} \textbf{Step 1: Rewrite $\bm{T_{N,f}(x,y)}$ locally in real coordinates} 

The goal is to write $T_{N,f} (x,y)$ for $x$ near $x_0 \in X$ and $y\in X$. We may assume $x_0 \in U_1$ with $\rho_1(x_0) = 0$ (recall $(U_i,\rho_i)$ are charts on $X$). By construction,
\begin{align}\label{eq.pg6} \begin{split}
 T_{N,f} (x,\bar y ) = \int_{U_1} \Pi_{N} (x,\bar w) f(w) &\Pi_N (w,\bar y) e^{-N\phi(w)} \dd\mu (w)  \\
+& \int_{X \setminus U_1} \Pi_{N} (x,\bar w) f(w)\Pi_N (w,\bar y) e^{-N\phi(w)} \dd\mu(w) . \end{split}
\end{align}
By Lemma \ref{lemma.global bergman kernel bound}, the second integral is $\exp(\frac{N}{2} (\phi (x) + \phi(y) ) )\mathcal{O} (N^{-\infty})$. If $y\notin U_1$, then by Lemma \ref{lemma.global bergman kernel bound}, \eqref{eq.pg6} will be $ \exp(\frac{N}{2} (\phi (x) + \phi(y) ) )\mathcal{O} (N^{-\infty})$. 

We now assume that $x,y\in U_1$ which allows us to work locally. For the remainder of the proof (until the last step) all computations are for $x$ and $y$ in this chart. We therefore replace $\rho_1(x)$ and $\rho_1(y)$ by $x$ and $y$ respectively and functions on $X$ are replaced by functions on $\C$ with the same name.

We now rewrite \eqref{eq.pg6} as
\begin{align}
\int_{\rho_1 ( U_1)} e^{N \Phi_{x,\bar y} (w)} g_{x,\bar y} (w) \dd m(w) + e^{ \frac{N}{2} (\phi (x) + \phi (y) ) } \mathcal{O} (N^{-\infty}) \label{eq.first reduction}
\end{align}
where
\begin{align}
\Phi_{x,\bar y}(w) &= \psi(x,\bar w) - \phi(w) + \psi (w,\bar y) ,\\
g_{x,\bar y}(w) &= \left ( \frac{N}{2\pi}\right ) ^{2d} f( w) B(x,\bar w) B(w,\bar y) \mu (w) \label{eq:g1} ,\\
\frac{\omega^{\land d}(w)}{d!} &= \mu(w) \dd m(w).
\end{align}
Here $\dd m(w)$ is the Lebesgue measure on $\C^d$. Note that locally if
\begin{align}
\omega = i \sum _{\ell,m = 1}^d H_{\ell,m} \dd w_\ell \land \dd  \bar w_m
\end{align}
then
\begin{align}
\omega^{\land d} / d! = 2^d \det (H) \dd \Re{w_1} \land \dd \Im{w_1} \land \cdots \land \dd \Re{w_d} \land \dd \Im{w_d}.
\end{align}
Recall that $H = \p \dbar \phi$ which is locally a positive definite matrix. Therefore, locally, $\mu(w) = 2^d \det (\p \dbar \phi(w) )$.
 
For any $M' \in \N$, the projector $\Pi_N$ can be estimated with $M'$ terms as in \eqref{eq:399}. Indeed, if we let $B(x,\bar y) := 1 + \sum _{1}^{M'} b_i (x,\bar y)$, as in \eqref{b_is}, then this introduces error $\mathcal{ O} (N^{2d - 2M' - 2} )\exp(\frac{N}{2} (\phi(x) + \phi(y)))$ which is absorbed into the error term in \eqref{eq.first reduction} as $M'$ can be arbitrarily large.

We next change to real coordinates by setting
\begin{align}
p&:=(\Re w ,\Im w ) \in \R^{2d},\label{eq:528} \\
t &: = (\Re x ,\Im x ,\Re y,- \Im y) \in \R^{4d}\label{eq:529}
\end{align}
 and define
\begin{align}
 \Psi (p,t) &: = \Phi_{x ,\bar y } (w (p) ) - \frac{1}{2} (\phi(x(t)) + \phi(y(t)) ) : \R^{2d} \times \R^{4d} \to \C , \label{eq:def Psi} \\
 g(p,t) &: = g_{x ,\bar y } ( w(p) )\chi (w(p)) : \R^{2d } \times \R^{4d} \to \C \label{eq:g2}
\end{align}
where $\chi\in C_0^\infty (\C^d; [0,1] )$ is identically $1$ near $0$. With this, the first term of \eqref{eq.first reduction} can be written
\begin{align} \label{one part rewritten}\begin{split}
 e ^{ \frac{N}{2} (\phi(x) + \phi(y))} \int _{\R^{2d}} e^{ N \Psi (p,t)} &g(p,t) \dd p  \\ 
 &  + \int_{\rho_1 ( U_1)} e^{N \Phi _{x,\bar y} ( w )} g_{x,\bar y} ( w) (1 - \chi(w) ) \dd m ( w ). \end{split}
\end{align}
By the Taylor expansion of $\Re{\Phi _{x,\bar y} (w)}$ stated in \eqref{eq.gaussian behaviour}, it immediately follows that the second term in \eqref{one part rewritten} is $ \exp(\frac{N} {2} ( \phi(x ) + \phi(y) ) ) \mathcal{O } (N^{-\infty})$.

\smallsection{Summary of step 1} We have observed that the Schwartz kernel of $T_{N,f}$, written $T_{N,f} (x,\bar y)$, is concentrated along the diagonal $y= x$. Near any $x$, we can approximate $T_{N,f}(x,\bar y)$ as an integral over $\R^{2d}$ of the form $\int \exp(N \psi(p,t)) g(p,t) \dd p$. Here $t$ is a function of $x$ and $y$, $g$ is a smooth compactly supported function depending on the symbol $f$, the Bergman kernel, and the density of the volume form on the K\"ahler manifold $X$, and $\Psi$ is a sum of phases appearing in the Bergman kernel. We would now like to apply the method of complex stationary phase to approximate this integral.

\noindent \textbf{Step 2: Deform the contour of the main term} 

Following the method of complex stationary phase presented by Tr\`eves in \cite[Chapter 10]{Treves}, the first term of \eqref{one part rewritten} will be estimated by a contour deformation. Let $\tilde \Psi $ and $\tilde g$ be almost analytic extensions of $\Psi$ and $g$ in the $p$ variable (as described in Propositions \ref{thm:aa e} and \ref{thm.extension of symbols}).

We first observe that there is a unique solution {$p_{\rm{crit} }(t)$} to $\p_p \tilde \Psi (p, t) = 0$ (where $p \in \C^{2d}$) such that the Hessian $\tilde \Psi _{pp} ({p_{\rm{crit} }(t)} , t)$ is invertible with real part negative definite and $\p_p$ is the holomorphic derivative in the $p$ variable. Indeed, by \cite[Proposition 2.2]{Hitrik}, $\Psi(p,0)$ has a unique critical point at $p = (0,0) $ with critical value equal to zero such that the real part of the Hessian is a negative definite matrix. By the implicit function theorem (see \cite[Chapter 10, Lemma 2.3]{Treves} for details), there exists a unique smooth function {$p_{\rm{crit} }(t)$} solving $\p_p \tilde \Psi ({p_{\rm{crit} }(t)}, t) = 0$. In Lemma \ref{lemma:p(t)}, an estimate of {$p_{\rm{crit} }(t)$} is proven.

The desired contour deformation relies on a particular function $q$, which is proven to exist in \cite[Chapter 10, Lemma 3.2]{Treves} and is stated here without proof.

\begin{lemma}\label{lemma.q_lemma}
There exist $U\subset \C^{2d}$, $V\subset \R^{4d}$ open neighborhoods of $0$ and a smooth function $q : U \times V \to \C^{2d}$ such that 
\begin{enumerate}
\item $q({p_{\rm{crit} }(t)}, t) = 0$,
\item for each $t\in V$, $p\mapsto q(p,t)$ is a diffeomorphism from $U$ onto an open neighborhood of zero in $\C^{2d}$,
\item $\tilde \Psi (p,t) - \tilde \Psi ({p_{\rm{crit} }(t)} , t) + \frac{1}{2} q(p,t) \cdot q(p,t) = \mathcal{O} ((\abs{\Im p } + \abs{\Im{{p_{\rm{crit} }(t)}}}) ^\infty ) $,
{\item $\p_p q(0,0) =  A_0$, where $A_0^t A_0 =- \Psi_{pp}(0,0)$.}
\end{enumerate} 
\end{lemma}

Let $q(p,t) = (z_1,\dots ,z_{2d}) = (x_1 + i y_1 ,\dots, x_{2d} + i y_{2d}) $. 
For each $t$, let $\mathcal{U} (t) := q ( \supp g \cap \R^{2d} , t) \subset \C^{2d}$. 
For $t$ close to $0$, there exists a function $\zeta$ such that $\mathcal{U}(t) = \set{x+ iy : y = \zeta (x,t) , x\in \mathcal{U}^\R (t)}$, where $\mathcal{U}^\R(t)$ is the projection of $\mathcal{U}(t)$ onto $\R^{2d}$.
{Indeed, because $\p_p q(0,0) = A_0$, tangent vectors  at the origin of $\mathcal U(0)$ are of the form $u = A_0 u_0$ for $u_0\in \R^{2d}$.
Because the real part of $A_0^t A_0$ is positive definite, $\Re{u \cdot u } \ge  0$.
Therefore if we write $u = v + i w$ (with $v,w$ real), then $|v|^2 \ge |w|^2$, so that $\mathcal{U}(0) = \set{x + i \psi(x) : x \in \mathcal U^\R (0)}$ for some function $\psi$.
For $t$ varying in a small neighborhood of $0$, we get the existence of $\zeta$ such that $\mathcal{U}(t) = \set{x+ i \zeta(x,t) : x\in \mathcal{U}^\R (t)}$.

}
 
For each $s \in [0,1]$, let $\mathcal{U}_s (t) = \set{x + is \zeta (x,t) : x\in \mathcal{U}^\R (t)}$. Define the contour:
 \begin{align}
 U_s(t) = \set{ p \in \C^{2d} : q(p,t) \in \mathcal{U}_s(t)} = q(\cdot , t) ^{-1 } (\mathcal{U}_s (t)).
 \end{align}
To ease notation below, let $z_s = z_s (x):= x + i s \zeta (x,t)$, and $p(z_s) = q(\cdot , t)^{-1} (z_s)$, so that $U_s = \set{ p (z_s (x) ) : x\in \mathcal{U}^\R}$. For a simple example of this contour construction, see Appendix \ref{appendix:treves}. For a schematic drawing of this contour construction, see Figure \ref{fig.1}.

Observe that $\mathcal{U} _1 (t) = \set{x + \zeta (x,t) : x\in \mathcal{U}^\R (t)} = \mathcal{U}(t) = q(\supp g \cap \R^{2 d } ,t )$. So 
\begin{align}
U_1 (t) = \set{p \in \C^{2d} : q(p,t) \in q(\supp g \cap \R^{2d} , t)} = \supp g \cap \R^{2d}.
\end{align}
Because this contains the support of $g$, we may rewrite the first integral in \eqref{one part rewritten}  as
\begin{align}
\int_{\R^{2d}}  e^{ N \Psi(p,t) } g(p,t) \dd p = \int_{U_1(t)} e^{N\tilde \Psi( p ,t)} \tilde g(p,t) \dd p,
\end{align}
which, by Stokes' theorem, is
\begin{align}
\int_{U_0(t)} e^{N\tilde \Psi( p ,t)} \tilde g(p,t) \dd p^1 \land & \cdots \land \dd p^{2d} \\
&+ \int _{W } (\dbar_p  (e^{N\tilde \Psi( p ,t) } \tilde g(p,t))) \land \dd p^1 \land \cdots \land \dd p^{2d}  \label{eq.Stokes Thm}
\end{align}
where $W = \set{ p\in \C^{2d} : q(p,t)\in \mathcal{U}_s (t), s \in [0,1] }$.
\begin{figure}[h!]
\includegraphics[width = \textwidth]{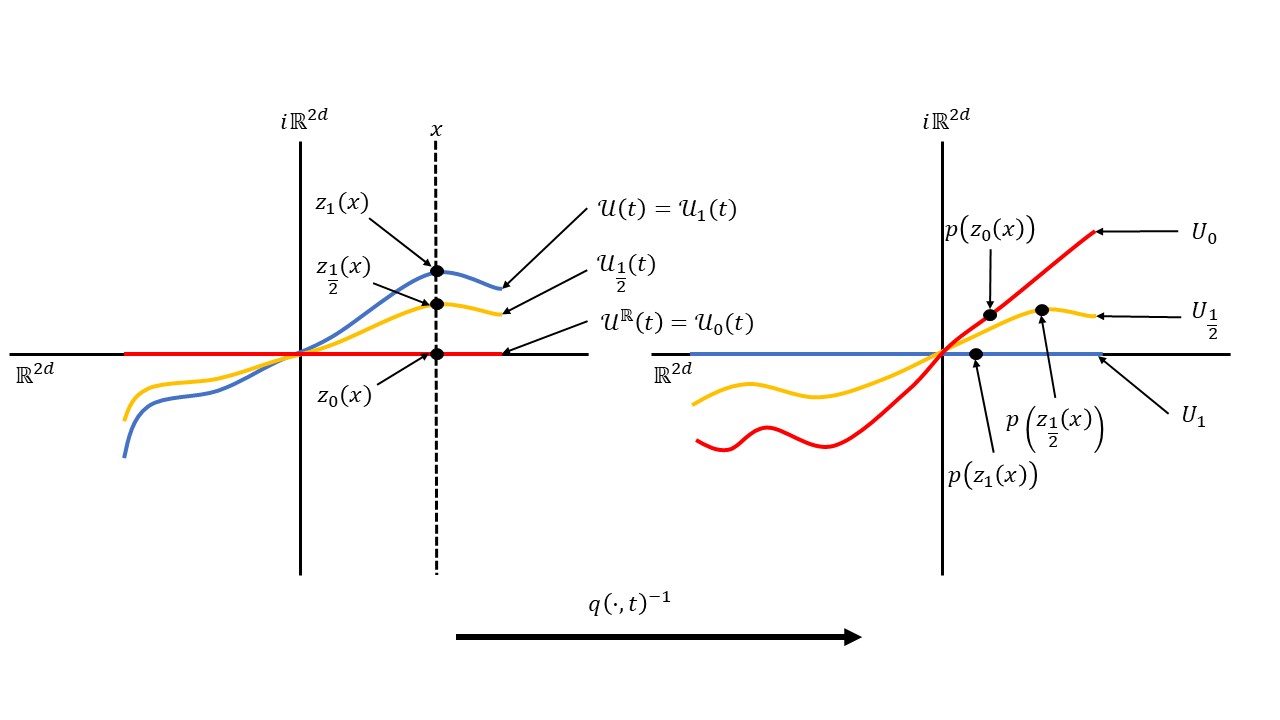}
\caption{Schematic of the contour deformation. \label{fig.1}}
\end{figure}

\smallsection{Summary of step 2} We considered $\int \exp(N \Psi ) g  \dd p $ as an integral over $\R^{2d}$ within $\C^{2d}$ by almost analytically extending $\Psi$ and $g$. We then chose a particular contour deformation and applied Stokes' theorem to rewrite this as two integrals (as in \eqref{eq.Stokes Thm}). We will show that by this choice of contour deformation and almost analytic extensions, the second term in \eqref{eq.Stokes Thm} is negligible.

\noindent \textbf{Step 3: Estimate $\bm{\dbar_p( \exp(N\tilde \Psi)\tilde g )} $} 

To control the second term of \eqref{eq.Stokes Thm}, $\dbar $-estimates for $\tilde \Psi$ and $\tilde g$ are required. Note that $ p \mapsto\Psi (p,t)\in C^\infty(\R^{2d}) $ (with uniform derivative estimates independent of $N$ as in Proposition \ref{thm.extension of symbols}). Therefore $\tilde \Psi$ has the usual $\dbar_p$ estimates (vanishing to infinite order in the imaginary direction). On the other hand, $g$ has growth in $N$ when differentiated.

Here we define $m_\R : \R^{2d} \to \R_{>0 }$ by
\begin{align}
m_\R (p_1,p_2 ) = m \circ \rho_1^{-1} (p_1 + i p_2). \label{eq:mR}
\end{align}
Then, by examining \eqref{eq:g1} and \eqref{eq:g2}, we see that $g(\cdot,t) \in S_\delta (m_\R)$ (prior to almost analytic extension). Therefore, by Proposition \ref{thm.extension of symbols}, for all $M \in \N$ and $p\in \C^{2d}$, we have that
\begin{align}
| \dbar_p \tilde g (p,t) |\lesssim_M N^{\delta - M \delta } \abs{\Im p }^M m_\R (\Re p ) .
 \end{align} 
Therefore:
\begin{align}
|\dbar _p( e^{N \tilde \Psi(p,t)} \tilde g(p,t))  |& = |(\dbar_p \tilde g(p,t) + N \dbar_p \tilde\Psi (p,t)) e ^{N \tilde\Psi(p,t) }| \\
& \lesssim_M e ^{N \Re{ \tilde\Psi(p,t) }} \left ( N ^{\delta - M \delta } m_\R (\Re p ) \abs{\Im p }^M + N \abs{\Im p }^M \right )\\
&\lesssim_M e ^{N\Re {\tilde \Psi(p,t)} } \abs{\Im p } ^M (N^{\delta (M+1)} m_\R(\Re p ) + N) . \label{eq.dbar F first} 
\end{align}
As we are integrating over $W$, we may write $p = p(z_s (x))$ for $x\in \mathcal{U}^\R$ and $s\in[0,1]$. Now a bound of $\exp(N \Re {\tilde \Psi (p(z_s(x)),t) } )$ is required. For this we apply the following Lemma.

\begin{lemma}\label{Lemma 3.4}
If $U$ and $V$ are small enough, there exists a $C > 0$ such that for all $(x,t) \in U\times V$ :
\begin{align}
\Re {\tilde \Psi (p(z_s (x)) ,t) }\le - C | \Im {p(z_s (x) }) |^2.
\end{align}
\end{lemma}
\begin{proof}
This is an upgraded version of \cite[Chapter 10, Lemma 3.2]{Treves}, which states that for all $s \in [0,1]$, there exists $C'>0$ such that
\begin{align}
\Re {\tilde \Psi (p(z_s (x)) ,t) }\le - C' (1-s)^{-1}| \Im {p(z_s (x) }) |^2- C' \abs{\Im{{p_{\rm{crit}}}(t)}}^2. \label{eq:601}
\end{align}
Observe that for $s=0$, $p(z_s(x)) \in \R^{2d}$. Then because $\Re { \tilde \Psi}$ has a unique critical point with negative definite Hessian, we see that for $s$ and $x$ near zero, there exists a $C '' >0$ such that $\Re{ \tilde \Psi (p(z_s (x))),t} \le -C'' |p(z_s(x)|^2\le  - C'' \abs{ \Im{p(z_s(x))}}^2$. We can combine this with \cite[Chapter 10, Lemma 3.2]{Treves} to get Lemma \ref{Lemma 3.4}.
\end{proof}

Using this lemma, we see that
\begin{align}
|\dbar_p (e^{N  \tilde\Psi(p,t)} g(p,t)) | &\lesssim_M e^{-CN \abs{\Im {p} }^2} \abs{\Im p }^M (N ^{\delta (M+1)} m_\R (\Re p ) + N ) \\
& \lesssim_{M,M'} N^{-M'} \abs{\Im{p}}^{M- 2M'} \max (N^{\delta (M+1 + M_0)} , N) 
\end{align}
for any $M'\in \N$. Here we used that $m_\R(x) \lesssim N^{\delta M_0 }$ and $ p $ is bounded on the region we are integrating. Let $M = 2M'$ so that
\begin{align}
|\dbar_p (e^{N \tilde\Psi(p,t)} g(p,t)) | \lesssim _M \max(N^{\delta ( M_0 +1) - M(1/2 - \delta)}, N^{1 - M/2}) = \mathcal{O}(N^{-\infty})
\end{align}
because $\delta < 1/2$, and $M$ can be made arbitrarily large.

\smallsection{Summary of step 3} In this step, we proved that the second term of \eqref{eq.Stokes Thm} is $\mathcal{O}(N^{-\infty})$. This involved estimating $\dbar_p$ applied to the integrand. This was controlled by Propositions \ref{thm:aa e} and \ref{thm.extension of symbols}. Because $\tilde g$ is a function of our symbol (so its derivatives are unbounded in $N$), these $\dbar_p$-estimates are weaker than the $\dbar_p$-estimates on $\tilde \Psi$. Fortunately, by the choice of contour deformation, on the domain of integration the phase behaves like a Gaussian, and destroys all temperate growth in the $\dbar_p$-estimates.

\noindent \textbf{Step 4: Reduction to quadratic phase}

We now compute the first term in \eqref{eq.Stokes Thm}. Define
\begin{align}
J(N): = \int_{U_0} e^{N  \tilde \Psi( p ,t)}  \tilde g(p,t) \dd p^1 \land \cdots \land \dd p^{2d}. \label{eq:628}
\end{align}
First change variables to integrate over $x\in \R^{2d}$:
\begin{align}
J(N) = \int _{\mathcal{U}^\R} e^{ N \tilde \Psi (p(x),t)}  \tilde g(p(x),t) \left (\frac{\p p}{\p x} \right ) \dd x.
\end{align}
Next, Taylor expand the phase about the critical point and interpolate the remainder. Define
\begin{align}
iR(x,t)  &: =  \tilde\Psi(p(x),t) - \tilde\Psi({p_{\rm{crit}}}(t) , t) + |x|^2/2,\\
\Psi_s (x,t)& := \tilde \Psi({p_{\rm{crit}}}(t),t) - |x|^2 /2 + i sR(x,t),\\
h(x) &: =  \tilde g(p(x),t) (\p p / \p x). \label{eq:638}
\end{align} 
Note that $\Psi_1(x,t) = \tilde \Psi(p(x),t)$ and $\Psi_ 0 (x,t)= \tilde \Psi({p_{\rm{crit}}} (t),t) - |x|^2/2$. We would like to prove that $J(N)$ can be estimated using the $\Psi_0$ phase with $\mathcal{O} (N^{-\infty})$ error. 

Using that
\begin{align}
\int_0^1 N R(x,t) e^{ N \Psi _s(x,t) }ds = e^{N\Psi_1(x, t) } - e^{ N \Psi _0 (x, t) },
\end{align}
we get that
\begin{align}
\left |\int _{\mathcal{U}^\R } \left ( e ^{ N \Psi_1(x,t) } - e ^{ N \Psi_0(p(x),t) } \right ) h (x) \dd x \right | = \left | \int_0^1 \int _{\mathcal{U}^\R} h(x) N R(x,t) e^{ N \Psi_s (x,t)}\dd x \dd s \right | \\
\le N \int_0^1 \int _{\mathcal{U}^\R} |h(x) R(x,t) | e^{N \Re {\Psi_s (x,t)}} \dd x \dd s. \qquad \label{eq.pg 9}
 \end{align}

We can control $R(x,t)$ by the following Lemma presented in \cite[Chapter 10, Lemma 3.2]{Treves}.

 \begin{lemma}
For all $x\in \mathcal{U}^\R(t)$, $t$ near $0$, $M >0$, there exist $C_M > 0$ such that $|R (x,t) | \le C _M (\abs{\Im { p(x) } } + \abs{\Im { {p_{\rm{crit} }(t)} } })^M $.
 \end{lemma}

Using this, and $\Re{ \tilde \Psi (p(x), t)} \le - C (\abs{\Im {{p_{\rm{crit} }(t)} } } ^2 + \abs{\Im { p(x) } }^2)$ (by \eqref{eq:601}), we see that for all $M \in \N$, there exists $C_M > 0 $, such that
\begin{align}
    |R(x,t) | \le -C_M\left  (\Re{\tilde \Psi (p(x) ,t) }\right )^M.
\end{align}

By expanding $\Psi_s(x,t)$, 
\begin{align}
 \Re{\Psi_s (x,t)} &= (1 - s) \Re{\tilde \Psi ({p_{\rm{crit} }(t)} ,t) } + (1-s ) \left ( \frac{-|x|^2}{2} \right ) + s \Re{\tilde \Psi (p(x),t) } \\
 &\le s \Re{\tilde \Psi (p(x),t) }.
 \end{align} 
Then, since $|h(x) | \le C  N ^{\delta M_0}$ (for some $C>0$), \eqref{eq.pg 9} is bounded by

\begin{align}
 C N^{\delta M_0+1} \Bigg ( \int _{\mathcal{U}^\R}  \Bigg ( \int_0^{1/2} \Big ( &-\Re{ \tilde\Psi ( p(x), t)}\Big) e^{\frac{-N}{4} |x|^2} \dd s   \\
  &+ \int_\frac{1}{2} ^{1} e^{s N \Re{\tilde \Psi (p(x) ,t ) } } \abs{\Re{\tilde \Psi (p(x) ,t )}}^M \dd s \Bigg ) \dd x \Bigg ).
\end{align}
Because $\mathcal{U}^\R$ is bounded, and $\Psi$ is bounded, both terms are $\mathcal{O} (N^{-\infty})$.

Therefore:
\begin{align}
J(N) = e^{N \tilde \Psi ({p_{\rm{crit} }(t)},t)} \int_{\mathcal{U}^\R} e^{-N|x|^2/2} g(p(x),t) \left ( \frac{\p p }{\p  x} \right ) \dd x + \mathcal{O} (N^{-\infty}) \label{step 4 result}
\end{align}

\smallsection{Summary of step 4} This step proceeded identically to \cite[Chapter 10]{Treves}. We began with our integral on the constructed contour $U_0$ in $\C^{2d}$ \eqref{eq:628}. We changed variables to integrate over the real variable $x\in \R^{2d}$. We proved that this integral can be approximated by replacing the phase by the critical value of the phase minus a quadratic term. This is to set us up to apply the saddle-point method (also called real stationary phase).

\noindent \textbf{Step 5: Apply the saddle-point method}

By the saddle-point method (see for instance \cite[Exercise 2.4]{Grigis}) for each $J \in \N$ we can now rewrite \eqref{step 4 result} as
\begin{align}
e^{N \tilde \Psi ({p_{\rm{crit} }(t)},t)} &\int_{\R^{2d}} e^{-N |x|^2/ 2} h(x)\dd x + \mathcal{O }(N^{-\infty}) \\
&=e^{N \tilde \Psi({p_{\rm{crit} }(t)},t)} \left (\frac{2\pi }{N} \right )^{d} \left ( \sum _{j=0}^{J-1 }\frac{1 }{ N^j j! 2^j} \Delta^j h (0) +N^{-J}R_J(t) \right ) \label{remaining part}
\end{align}
with error bound
\begin{align}
|R_J (t) | \lesssim_J  \sum _{|\alpha | = 2(J+1)} \sup _{x \in \R^{2d}}| \p ^\alpha h(x)| ,\label{eq.remainder term}
\end{align}
where $h$ is defined in \eqref{eq:638}.

We now have to unravel all the definitions of the functions in \eqref{remaining part}. First, replace $t$ by $(x,\bar y)$. From the first four steps, we have shown that for $x,y$ near zero, in local coordinates, $\exp(-(N/2)(\phi(x) + \phi(y)) T_{N,f}(x,\bar y)$ is
\begin{align}
e^{ N\tilde \Psi ({p_{\rm{crit}}}(x,\bar y),x, \bar y)} \left ( \frac{N}{2\pi} \right )^d \left ( \sum_{j=0}^{J-1} N^{-j} f_j (x,\bar y) + N^{-J}R_J(x,\bar y) \right ) + \mathcal{O} (N^{-\infty}),\label{eq:623}
\end{align}
for each $J \in \N$, where:
\begin{align}
f_j (x,\bar y) = ( j! 2^{j})^{-1} \Delta _z ^j h(x,\bar y, 0)\label{eq.fj}
\end{align}
with
\begin{align}
 h(x,\bar y, z) = \tilde f(p(x,\bar y,z) ) \tilde g_2 (x,\bar y, p (x,\bar y, z) ) \det \left ( \frac{\p p (x,\bar y, z)}{\p z} \right ) .\label{eq.h definition}
 \end{align} 
 
Here we are defining $p (x,\bar y,\cdot) := q^{-1}(\cdot, x,\bar y  ) $ (with $q$ the change of variables defined in Lemma \ref{lemma.q_lemma}). As usual, the derivatives of the terms in the stationary phase expansion are evaluated at the critical point of the (almost analytically extended) phase. Indeed,
\begin{align}
p(x,\bar y, 0 ) = q^{-1}(\cdot, x,\bar y ) (0) = {p_{\rm{crit}}} (x, \bar y)
\end{align}
by the first property of $q$ in Lemma \ref{lemma.q_lemma}.

Recall that $x,y\in \C^d$, $z \in \R^{2d}$, $p\in \C^{2d}$, and $\tilde g_2$ is an almost analytic extension of $g$ defined in the following way. We let
\begin{align}
g_2 (w,x,\bar y) := B(x,\bar w)B(w,\bar y) \det (\p \dbar \phi (w) )\chi(w) :  \C^d _w \times \C^d_x \times \C^d_y \to \C
\end{align}
where $B(\cdot,\cdot)$ comes from the Bergman kernel expansion, $\phi$ is the K\"ahler potential and $\chi$ is a smooth cut-off function. Then we let $ p = (\Re w, \Im w)\in \R^{2d}$, and define
\begin{align}
(g_2)_\R  (p,x,\bar y) : = g_2 (w(p),x,\bar y) : \R^{2d}_p\times \C^d_x\times \C^d_y\to \C
\end{align}
and finally let $\tilde g_2$ be the almost analytic extension of $(g_2)_\R$ in the $p$ variable.

By the support property of almost analytic extensions, we can choose an $\e > 0 $ such that $\tilde g_2 (x,\bar y, p (x,\bar y , z)) = 0$ if $|p (x,\bar y ,  z)|  > \e$.

Also observe that when taking derivatives of $h$ with respect to $z$, everything is uniformly bounded in $N$, except when derivatives fall on $\tilde f (p) $. Therefore by Proposition \ref{thm.extension of symbols}, for any $\alpha\in \N^{4d}$ and $j\in  \Z_{\ge 0}$:
\begin{align}
|(\p^{\alpha}_{x,\bar x,y,\bar y} f_j )(x,\bar y) | \lesssim_{\alpha,j} N^{2\delta j + |\alpha | }  m_\R (\Re{{p_{\rm{crit}}}(x,\bar y) }), \label{eq:646}
\end{align}
with $m_\R$ defined in \eqref{eq:mR}.

\smallsection{Summary of step 5} In this step we applied the saddle-point method to obtain an asymptotic expansion of the Schwartz kernel of $T_N f$. To show that this asymptotic expansion makes sense ($f_j$'s belong to appropriate symbol classes and the remainder is controlled) we have to compute derivatives of the terms in the expansion. These terms are almost analytic extensions of functions whose derivatives are unbounded in $N$. However, we can see that they are bounded by powers of $N$ times the order function evaluated at the critical point of the phase, ${p_{\rm{crit}}}(x,\bar y)$. We must now estimate ${p_{\rm{crit}}}(x,\bar y)$ (this will also be used in estimating $\exp(N \tilde \Psi({p_{\rm{crit}}}(x,\bar y), x, \bar y)$).

\noindent \textbf{Step 6: Estimate critical value of phase}

Recall that for each $x,y\in \C^d$ near $0$, ${p_{\rm{crit}}}(x,\bar y)$ is the unique $p \in \C^{2d}$ such that
\begin{align}
\p_p \tilde \Psi(p,x,\bar y ) = 0
\end{align}
where $\p_p$ is the holomorphic derivative in the $p$ variable, and $\tilde \Psi(p,x,\bar y)$ is an almost analytic extension of $\Psi(p,x,\bar y)$ (defined in \eqref{eq:def Psi}) in the $p$ variable.

The goal is to show that ${p_{\rm{crit}}}(x, y)$ is 
\begin{align}
\frac{1}{2} ((x+ y), -i(x-y)) + \mathcal{O}(|x- \bar y|^\infty ) .\label{eq:734}
\end{align}
This would follow immediately, with no error, if we considered real-analytic K\"ahler potentials $\phi$.  We will show that ${p_{\rm{crit}}}(x,y)$ is almost analytic off of $y = \bar x$, and coincides with $\frac{1}{2} ((x+ y), -i(x-y)) $ on $y = \bar x$. By uniqueness of almost analytic extensions (modulo appropriate error), this will imply that ${p_{\rm{crit}}}(x,y)$ is \eqref{eq:734}.
 
\begin{lemma}\label{Lemma:627}
If $|x-y|$ is sufficiently small, there exists a constant $C>0$ such that:
\begin{align}
|{p_{\rm{crit},1}} (x,\bar y) + i {p_{\rm{crit},2}}(x,\bar y) - x | \le C |x - y | ,
\end{align}
where ${p_{\rm{crit}}} = ({p_{\rm{crit},1}},{p_{\rm{crit},2}}) \in \C^d \times \C^d$.
\end{lemma}
\begin{proof}
By Taylor expansion, it suffices to show that ${p_{\rm{crit}}}(x,\bar x) =(\Re x , \Im x )$. Before extension, $\p_p \Psi (p_1,p_2, x,\bar x) = 0$ has the unique solution $p_1 = \Re x, p_2 = \Im x$. Observe that for any $k \in C_0^\infty(\R; \C)$ with an almost analytic extension $\tilde k$, and any $z_0 \in \C$ with $\Im {z_0} = 0 $,   
\begin{align}
 (\p _z \tilde k)(z_0) = (\p_{\Re z} \tilde k )(z_0) = (\p  k )(\Re{z_0})
 \end{align} 
by Theorem \ref{thm:aa e}.

Applying this observation to $\tilde \Psi$ and letting $\tilde p_1$ be the extension of $p_1$ from $\R^d$ to $\C^d$, we see that
\begin{align}
\frac{1}{2} \p _{\tilde p _1} \tilde \Psi (\Re x, \Im x, x,\bar x )&= \p_{\Re{ \tilde p_1}} \tilde \Psi (\Re x, \Im x,x,\bar x)\\
& = \p_{p_1} \tilde \Psi (\Re x , \Im x , x,\bar x) = 0
\end{align}
and similarly for $\p _{\tilde p_2}$. Because ${p_{\rm{crit}}}$ is unique, the claim is proven.
\end{proof}

\begin{lemma}\label{lemma.confusing} If $|x-y|$ is sufficiently small, then
\begin{align}
    \dbar _x {p_{\rm{crit}}}(x,y) = \mathcal{O} (|x-y|^\infty) && \text{and} && \dbar _y {p_{\rm{crit}}} (x,y) = \mathcal{O} (|x-y|^\infty).
\end{align}
\end{lemma}

Before proving Lemma \ref{lemma.confusing}, we give some brief remarks. Proving this lemma is relatively confusing partially due to non-optimal notation (however we try to present a proof as clearly as possible). The difficulty is that we require $\dbar$ estimate of an almost analytic extension of a function that has been almost analytically extended. 

\begin{figure}[h!]
\[\begin{tikzcd}
	& {\tilde \Psi( p,x,\bar y)} \\
	\\
	& {\Psi(p,x,\bar y)} \\
	{\psi(x,\bar y)-\frac 12(\phi(x)+\phi(y))} && {\psi(x,\bar y)-\frac12(\phi(x)+\phi(y))} \\
	& 0
	\arrow["{p\in \R^{2d}}"', from=1-2, to=3-2]
	\arrow["{p=(\Re x, \Im x)}"', from=3-2, to=4-1]
	\arrow["{x=y}"', from=4-1, to=5-2]
	\arrow["{x=y}", from=4-3, to=5-2]
	\arrow["{p=(\Re y , \Im y)}", from=3-2, to=4-3]
\end{tikzcd}\]
\caption{Diagram of the restrictions of $\tilde \Psi$ to totally real vector spaces. First, we have $\tilde \Psi(p,x,\bar y) \in C^\infty (\C^{2d}_p \times \C^d_x \times \C^d_y)$ . This is almost analytic off of $p\in \R^{2d}$, whose restriction to $p\in \R^{2d}$ is $\Psi (p,x,\bar y) \in  C^\infty (\R^{2d}_p \times \C^d_x \times \C^d_y)$. We can restrict $\Psi$ to either $p = (\Re x, \Im x)$ or $p = (\Re y, \Im y)$ to get the same function as shown. When either of these functions are restricted to $x = y$ we get the zero function. Understanding various $\dbar$ estimates on $\tilde \Psi$ is the core part of this step of the proof. \label{fig.2}}
\end{figure}
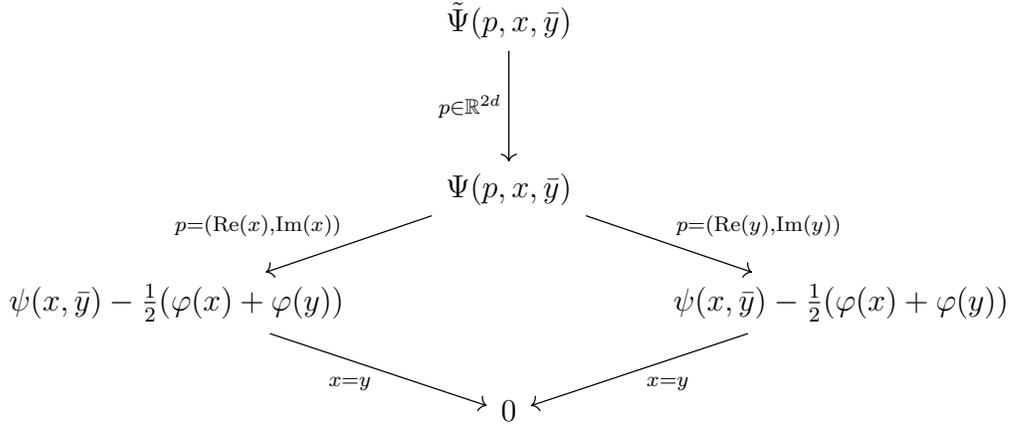

The core of the proof is to show that various $\dbar$ estimates of $\tilde \Psi$ rapidly decay as $|x-\bar y|$ goes to zero. Recall the construction of $\tilde \Psi$. We began with $\phi$ (the K\"ahler potential), then we almost analytically extended it to $\psi$ such that $\psi(x,\bar x) = \phi(x)$, then we defined $\Psi(p,x,\bar y) = \psi(x,\bar p) - \phi (p) + \psi(p,\bar y) - \frac{1}{2} (\phi(x) + \phi (y) ) $, then we almost analytically extended this in the $p$ variable. See Figure \ref{fig.2} for a schematic diagram of these extensions.

One example to keep in mind is the K\"ahler manifold $\C$ with symplectic form $\omega = i \dd x \land \dd \bar x$ with K\"ahler potential $\phi(x) = |x|^2$. In this case $\psi(x,\bar y)  = x\bar y$ (which is unique), so that $\Psi(p,x,\bar y) = x\bar p - |p|^2 +p\bar y - \frac{1}{2} (|x|^2 + |y|^2)$. The (unique) almost analytic extension of this in the $p$ variable is $\tilde \Psi ( p ,x,\bar y) = x(p_1 - i p_2)  - p_1^2 - p_2^2 + (p_1 + i p_2) \bar y - \frac{1}{2} (|x|^2 + |y|^2)$ (where $p = (p_1,p_2) \in \C^d \times \C^d$).

The key properties of $\tilde \Psi$ to use is that $\tilde \Psi(p,x,\bar y) = \Psi(p,x,\bar y)$ when $p\in \R^{2d}$ and $\psi(x,\bar x) = \phi(x)$ (ie the extensions agree on certain totally real vector spaces). By Taylor expanding from these totally real vector spaces, we prove Lemma \ref{lemma.confusing}.

\begin{proof}
Recall the chain rule for holomorphic differentiation:
\begin{align}
\dbar _z (f(g(z)) = (\p_z f)(g(z)) \cdot (\dbar_z  g)(z) + (\dbar _z  f)(g(z)) \cdot (\dbar_z \bar g )(z)
\end{align}
for arbitrary $f,g\in C^\infty (\C)$.

We can use this when computing $\dbar _{x_j} \p_p \tilde \Psi$ in conjunction with the implicit function theorem, to see that
\begin{align}
\dbar _{x_j} {p_{\rm{crit}}} (x,y) = - (\p_p \p_p \tilde \Psi )^{-1}  \left ( (\dbar _{x_j} \p_p\tilde  \Psi) + (\dbar_p \p_p \tilde \Psi )\dbar _{x_j} \overline{ {p_{\rm{crit}}}}(x,y)  \right )\label{eq.pg10}
 \end{align}
where all derivatives of $\tilde \Psi$ are evaluated at $({p_{\rm{crit}}}(x,y),x,y )$. The term $(\p_p \p_p \tilde \Psi )^{-1}$ is uniformly bounded for $x,y$ close to zero. We now claim the following $\dbar\tilde  \Psi $ estimates at $(\tilde w_1, \tilde w_2, x,y)$ for $\tilde w_1,\tilde w_2 \in \C^d$:
\begin{align}
\dbar_x \tilde \Psi &= \mathcal{O} (( \abs{\Im {\tilde w_1}}+ \abs{\Im {\tilde w_2} } + \abs{x - \Re{\tilde w_1} - i \Re{\tilde w_2}}) ^\infty),\label{eq.pg11} \\
\dbar_{y} \tilde \Psi &= \mathcal{O} (( \abs{\Im {\tilde w_1}} + \abs{\Im {\tilde w_2} } + \abs{y - \Re{\tilde w_1} + i \Re{\tilde w_2}}) ^\infty),\\
 \dbar_{{p_{\rm{crit}}}} \tilde \Psi &= \mathcal{O} ((\abs{\Im{\tilde w_1} }+ \abs{ \Im{\tilde w_2} })^\infty) .
\end{align}
Here we prove \eqref{eq.pg11}. Before extension, the only term in $\Psi$ depending on $x$ is $\psi(x,\bar w)$ (recall $\psi$ is an almost analytic extension of $\phi$ as described in Subsection \ref{subsection:Bergman kernels}). In real coordinates, this is $\psi(x_1,x_2,w_1, - w_2)$. An almost analytic extension in the $w$ variable can be written $\tilde \psi(x_1,x_2,\tilde w_1, -\tilde w_2) : \C^{4d} \to \C$. To compute $\dbar _x \tilde \psi$, we Taylor expand about $(x_1,x_2,\Re{\tilde w_1}, - \Re{\tilde w_2})$ with $K\in \N$ terms, so $\dbar_{x} \tilde \psi(x_1,x_2,\tilde w_1,\tilde w_2)$ is
	\begin{align}	
\begin{split}	
\frac{1}{2} (\p_{x_1} + i \p_{x_2} ) \left ( \sum _{|\alpha | \le K} \frac{\p^\alpha _{w_1, w_2} \tilde \psi(x_1,x_2,\Re{\tilde w_1} , - \Re{\tilde w_2} )}{\alpha!} (i \Im{\tilde w_1}, i \Im{\tilde w_2})^\alpha  \right.\\	
\left. + \vphantom{\sum_\alpha \frac{2}{1}} R_K(x_1,x_2,\tilde w_1,\tilde w_2) \right ) .	
\end{split}	
\end{align}
The $\dbar _x$ operator can be commuted with the $\p_{w_1,w_2}^\alpha$ operator by Proposition \ref{thm:aa e}. At $(x_1,x_2,\Re{\tilde w_1},\Re{\tilde w_2})$, $\tilde \psi = \psi$, and the $\dbar$ estimates can be made uniform with respect to differentiation of $w_i$. Recalling how $\psi$ is an almost analytic extension of $\phi$ provides the estimate: $(\p_{x_1} + i \p_{x_2}) \psi(x_1,x_2,y_1,y_2) = \mathcal{O} (|(x_1 + i x_2) - (y_1 - i y_2) |^\infty)$, so that for each $\alpha$:
\begin{align}
\frac{1}{2\alpha!} \dbar_x \p^\alpha _{w_1,w_2} \tilde \psi (x_1,x_2, \Re {\tilde w_1} , - \Re{\tilde w_2 }) (i \Im{\tilde w_1}, i \Im{\tilde w_2})^\alpha \\
= \abs{\Im {\tilde w} }^{|\alpha |} \mathcal{O } (\abs{x - \Re {\tilde w_1 } - i \Re {\tilde w_2 } }^\infty ) 
\end{align}
while $\dbar _x R_k = \mathcal{O}(\abs{\Im{\tilde w})}^{K+1} ) $. This proves \eqref{eq.pg11} and the others follow similarly.

Because $\Im {{p_{\rm{crit}}}(x,\bar x ) } =0$, by Taylor expansion: $\abs{\Im{ {p_{\rm{crit}}}(x,y) }}^2 \le C |x-\bar y|^2$ for some $C> 0$. By Lemma \ref{Lemma:627}, $|x - \Re{{p_{\rm{crit,1}}}(x,y)} - i \Re{{p_{\rm{crit},2}}(x,y)} |\le C |x -\bar  y|$. 
Using this, and the estimate of $\dbar _x \tilde \Psi$, we see that $\dbar _x \tilde \Psi ({p_{\rm{crit}}}(x,y),x,y) = \mathcal{O} (|x -\bar y|^\infty)$. 
This is true for all terms on the right side of \eqref{eq.pg10}, so that $\dbar_x {p_{\rm{crit}}}(x,y) = \mathcal{O} (|x- \bar y|^\infty)$. By an identical argument, $\dbar_y {p_{\rm{crit}}}(x,y) = \mathcal{O} (|x -\bar y|^\infty)$.
\end{proof}

\begin{lemma}\label{lemma:p(t)}
For $|x-\bar y|$ sufficiently small
\begin{align}
{p_{\rm{crit}}}(x,y) = \left ( \frac{1}{2} (x+ y) , \frac{1}{2i} (x-y) \right ) + \mathcal{O}(|x- \bar y| ^\infty).
\end{align}
\end{lemma}
\begin{proof}
From Lemma \ref{lemma.confusing}, we have that ${p_{\rm{crit}}} $ is almost analytic off the diagonal $y =  \bar x$. The function $(x,y) \mapsto ( 2^{-1} (x+ y) , (2i)^{-1} (x-y) ) $ is holomorphic and agrees with ${p_{\rm{crit}}} $ on $y =\bar  x$. Therefore by uniqueness (modulo $\mathcal{O}(|x- \bar y| ^\infty) $ error) of almost analytic extensions, the lemma follows.
\end{proof}

\smallsection{Summary of step 6} In this step we provided an estimate of ${p_{\rm{crit}}} (x,\bar y)$, the critical point of the phase $\tilde \Psi(p,x,\bar y)$. We will now use this to provide derivative estimates of the terms coming from the stationary phase expansion in \eqref{eq:646}.

\noindent \textbf{Step 7: Prove symbol estimates of stationary phase terms}

A simple computation shows that from Lemma \ref{lemma:p(t)},
\begin{align}
\Re {{p_{\rm{crit}}}(x,y) }_1 + i \Re{{p_{\rm{crit}}}(x,y)}_2 = \frac{1}{2} (x + \bar y) + \mathcal{O} (|x - \bar y| ^\infty),\label{eq:965}
\end{align}
where for $a,b\in \C^d$, $c = (a,b)\in \C^{2d}$, we write $c_1 = a$ and $c_2 = b$.

Then, recalling the definition of $m_\R$ from \eqref{eq:mR} and using \eqref{eq:965}, we get that
\begin{align}\label{eq:708}\begin{split}
m_\R (\Re{{p_{\rm{crit}}}(x,\bar y) }) &= m\left ( \frac{1}{2} (x+y) + \mathcal{O} (|x-y|^\infty)\right ) \\
&\lesssim m(x) (1 + N^\delta |x- y| ) ^{M_0}.\end{split}
\end{align}
Applying \eqref{eq:708} to the derivative estimate of $f_j$ in \eqref{eq:646}, we get that
\begin{align}
f_j (x,\bar y) \in N^{2\delta j } S_\delta (m_\R({p_{\rm{crit}}}(x,\bar y) )) \subset N^{2\delta j } S_\delta ( m(x) (1+ N^\delta |x-y| )^{M_0} ) . \label{eq:977}
\end{align}
We similarly have that $|{p_{\rm{crit},1}} (x,\bar y) + i {p_{\rm{crit},2}} (x,\bar y) - y | \le C |x - y | $ (for some $C>0$), so that
\begin{align}
f_j (x,\bar y) \in N^{2\delta j } S_\delta ( m(y) (1+ N^\delta |x-y| )^{M_0} ).
\end{align}

The supports of these $f_j$'s are contained in a strip along the diagonal, shrinking with respect to $N$. Indeed, because $p(x,\bar y, 0 ) ={p_{\rm{crit}}}(x,\bar y) = (2^{-1} (x + \bar y), (2i)^{-1} (x - \bar y) ) + \mathcal{O} (|x - y|^\infty) $ we get that $\abs{\Im{p(x,\bar y,0)}} \lesssim |x- y|$. Then observe that in \eqref{eq.h definition}, the term $\tilde f (p)$ is included, which by Proposition \ref{thm.extension of symbols}, is supported where $\abs{\Im{ p} } \lesssim N^{-\delta}$. Therefore, there exists $C > 0$ such that
\begin{align}
\supp f_j (x,\bar y) \subset \set{|x -y | \le C N^{-\delta}}. \label{eq:986}
\end{align}
But now we can apply \eqref{eq:986} to \eqref{eq:977} to see that:
\begin{align}
f_j(x,\bar y) \in N^{2\delta j } \left (  S_\delta (m(x))\cap S_\delta (m(y))  \right ).
\end{align}

The remainder can be bounded similarly. For each $\alpha\in \N^{4d}$
\begin{align}
|(\p^\alpha _{x,\bar x, y,\bar y} R_J)(x,\bar y) | &\lesssim_{\alpha,J}  \sum _{|\beta | = 2J} \sup _{z } |(\p^\alpha _{x,\bar x,y,\bar y} \p^\beta _z h)(x,\bar y ,z)| \\
&\lesssim_{\alpha,J}  N^{2J\delta } N^{|\alpha | \delta } \sup _{z\in \supp h(x,\bar y, \cdot )} m_\R (\Re{p(x,\bar y,z)}).
\end{align}
Now $\tilde g_2$ (defined in \eqref{eq.h definition}) is only supported where $|p| < \e$ so there exists $C>0$ such that 
\begin{align}
\sup _{z\in \supp h(x,\bar y, \cdot )} m_\R (\Re{ p(x,\bar y,z)}) \le C \min( m(x),m(y) ) (1 + N^\delta \e )^{M_0},
\end{align}
therefore
\begin{align}
R_J(x,\bar y) \in N^{2\delta J + \delta M_0 } ( S_\delta (m(x)) \cap S_\delta (m(y)) ). \label{eq:R_J bound}
\end{align}
We can bootstrap this to prove the better remainder bound stated in the Theorem. For any $J\in \Z_{> 0 }$, we rewrite the sum and remainder in \eqref{eq:623} as
\begin{align}
\sum_{j=0}^{J-1} N^{-j} f_j(x,\bar y) + N^{-J} \left ( \sum_{j=J}^{\tilde J-1} N^{J-j} f_j (x,\bar y) + N^{J - \tilde J} R_{\tilde J} (x,\bar y) \right )\label{eq:1008}
\end{align}
where $\Z \ni \tilde J > J + (\delta M_0)  (1-2\delta)^{-1} $. This choice of $\tilde J$ ensures that $N^{J-\tilde J}R_{\tilde J} (x,\bar y)$ belongs to $N^{2\delta J}(S_\delta (m(x)) \cap S_\delta (m(y)))$. It is also clear that for each $j = J , \dots, \tilde J - 1$, $N^{-j + J} f_j(x,\bar y) \in N^{2\delta J} (S_\delta (m(x)) \cap S_\delta (m(y)))$. We can therefore define the terms multiplied by $N^{-J}$ in \eqref{eq:1008} as the remainder term $R_J(x,\bar y)$ stated in the Theorem.


\smallsection{Summary of step 7} We showed in step 5 that derivatives of terms in the stationary phase expansion in \eqref{eq:623} are bounded in terms of the order function of $f$ evaluated at the critical point of the almost analytically extended phase, ${p_{\rm{crit}}} (x,\bar y)$ (see \eqref{eq:646}). From step 6, we estimated ${p_{\rm{crit}}} (x,\bar y)$ (Lemma \ref{lemma:p(t)}) to provide more explicit symbol estimates for $f_j$'s, and the remainder terms, in the stationary phase expansion. We now have a local expansion of the Schwartz kernel of the Toeplitz operator $T_N f$. To prove symbolic calculus results, we need to show this expansion is unique (modulo appropriate error) and the principal term is the principal term of $f$.

\noindent \textbf{Step 8: Prove stationary phase terms are almost analytic off the diagonal}

\begin{lemma}
We may choose an almost analytic extension of $\Psi$ such that {
$\tilde \Psi (p_{\text{crit}}(t),t)$ in \eqref{eq:623} can be replaced by $\psi(x,\bar y)- 2^{-1}(\phi (x) + \phi (y))$. In particular
}
$$\Psi ({p_{\rm{crit}}}(t), t)  = \psi (x,\bar y) - 2^{-1} (\phi(x) + \phi(y) ) {+\mathcal{O}(\dist(x,y)^\infty)}.$$
\end{lemma}
\begin{proof}
Because the Toeplitz quantization of the identity is the Bergman kernel, the phase can be recovered up to an appropriate error. 
Recall that \cite{BBS} showed:
\begin{align}
e^{-\frac{N}{2} (\phi(x)+\phi(y))}T_{N,1} (x,\bar y ) \sim e^{N( \psi (x,\bar y) - \frac{1}{2} (\phi(x) + \phi(y)) } \left ( \frac{N}{ 2 \pi} \right )^{d}\sum_0^\infty N^{-j} b_j (x,\bar y).
\end{align}

This must agree, up to $\mathcal{O} (N^{-\infty})$ error, with \eqref{eq:623}. It is therefore possible to choose $\tilde \Psi$ such that $\tilde \Psi ({p_{\rm{crit}}},t) = \psi(x,\bar y) - 2^{-1} (\phi(x) + \phi(y))$. 
\end{proof}

\begin{lemma}
All the $f_j (x,y)$'s are almost analytic off of $y = \bar x$.
\end{lemma}
\begin{proof}

When computing $(\dbar_{x,y} \Delta _z ^j) h(x,\bar y,0)$, observe the following properties of the functions making up $h$ in \eqref{eq.h definition}:
\begin{enumerate}
\item When $\dbar _x$ falls on $\tilde f$, we get $(\p _p \tilde f) (\dbar _x {p_{\rm{crit}}}) + (\dbar _p \tilde f)( \dbar_x \overline{ {p_{\rm{crit}}}})$. The first term is controlled by almost analyticity of ${p_{\rm{crit}}}$ off of $y = \bar x $ while the second term is controlled by almost analyticity of $\tilde f$ off of ${p_{\rm{crit}}}\in \R^{2d}$.
\item When $\dbar _x$ falls on $\tilde g_2$, we get $(\dbar_1 \tilde g_2) +  (\p_3 \tilde g_2)( \dbar _x {p_{\rm{crit}}} )+ (\dbar _3 \tilde g_2)( \dbar \overline{{p_{\rm{crit}}}})$ where $\p_i$ and $\dbar_i$ are the holomorphic and anti-holomorphic derivatives of the $i^{th}$ argument of $g$ respectively. The first term is controlled by almost analyticity of $B(x,\bar w)$ off of $w = x$, the second term is controlled by almost analyticity of ${p_{\rm{crit}}}$ off of $y = \bar x$, and the third term is controlled by almost analyticity off both totally real manifolds.
\item When $\dbar _x$ falls on the determinant, we get control by the almost analyticity of ${p_{\rm{crit}}}$ off of $y = \bar x$.
\end{enumerate}

Therefore
\begin{align}
\dbar _x f_j (x,y) = N ^{\delta (1+2j + M_0)} \mathcal{O} ( (|x - y| N^\delta)^\infty). \label{almost almost aa}
\end{align}
\end{proof}

Note that in this expansion, only knowledge of the kernel along the diagonal is required. Indeed, if $f_j$ and $g_j$ agree along $y = x$, and obey \eqref{almost almost aa}, then by the Gaussian behavior of the phase:
\begin{align}
\left ( \frac{N}{2 \pi}\right )^d e^{N\psi(x,\bar y)} \sum N^{-j} (f_j (x,\bar y) - g_j (x,\bar y) ) = e^{\frac{N}{2} (\phi(x)+ \phi(y)) }\mathcal{O} (N^{-\infty}).
\end{align}

\smallsection{Summary of step 8} Here we show the terms coming from the stationary phase expansion are almost analytic off the totally real vector space $x = y$ which provides a unique expansion (modulo appropriate error). The final step of the proof is to compute the first term along the diagonal, and prove the global statement.

\noindent \textbf{Step 9: Zeroth order term and global statement}

Examining \eqref{eq.fj} and the subsequent equations, along the diagonal
\begin{align}
f_0 (x,\bar x) = f(x) B(x,\bar x) B(x,\bar x) \det (\p \dbar \phi (x) ) \det \left ( \frac{\p q^{-1} (x,\bar x, z)}{\p z} \right ) \Big |_{z = 0}.
\end{align}
This can all be explicitly computed. But note that nothing, except $f$, on the right-hand side depends on $f$. And if $f = 1$, then $T_{N,f} = \Pi_N$. By \cite{BBS}, the leading order term of $\Pi_N$ is $1$, therefore, we know that everything on the right-hand side of order $N^0$, except $f$, must be $1$. Therefore $f_0 (x,\bar x) = f(x)+ \mathcal{O}(N^{-(1-2\delta)} m(x) )$. In the appendix, the second term is computed.

We now have proven existence of $f_j $ locally, in a ball of radius $\e$, around any point $x\in X$. Because each $f_j$ is unique along the diagonal, we can patch together $f_j's$ to construct a global $f_j$ defined near the diagonal. 
\end{proof}

\subsection{Composition of Toeplitz Operators}
Suppose that $f \in S_\delta (m_1)$ and $g\in S_\delta (m_2)$ for two $\delta$-order functions $m_1$ and $m_2$ and $\delta \in [0,1/2)$. Roughly, this section constructs a function $h \in S_\delta (m_1 m_2) $ such that $T_{N,f} \circ T_{N,g} \approx T_{N,h}$. This $h$ will be written as a star product: $h =f \star g$ following the now standard notation first introduced in \cite{Bayen}.

\begin{theorem}[\textbf{Composition Estimate}]\label{Theorem:composition }
For $\delta\in [0,1/2)$, suppose $m_1$ and $m_2$ are two $\delta$-order functions on $X$ (a quantizable K\"ahler manifold), $f \in S_\delta (m_1)$ and $g\in S_\delta (m_2)$.  Then there exists $(f \star g ) \in S_\delta (m_1 m_2)$ such that
\begin{align}
\norm{T_{N,(f\star g )} - T_{N,f} \circ T_{N,g}}_{L^2(X,L^N) \to L^2(X,L^N)} = \mathcal{O}(N^{-\infty})
\end{align}
{and 
\begin{align}
    (f\star g )(x) = f(x) g(x) +\mathcal{O}(N^{1-2\delta } m_1(x)m_2(x)).
\end{align}
}
\end{theorem}
{
The construction of $f\star g$ in Theorem \ref{Theorem:composition } follows the main argument presented in \cite{Charles} for the composition of classic Berezin--Toeplitz operators.
This closely follows the derivation of star product for the Weyl quantization of functions on $\R^{2d}$ (see for instance \cite[\S 2.3]{folland1989harmonic} or \cite[Theorem 4.12]{Zworski} among others).
}
\begin{proof}
\noindent {\textbf{Step 1:}}
Let $h = f \star g \in S_\delta (m_1m_2)$ {(which will be constructed in Step 3 of this proof)} and let $K(x,y)$ be the Schwartz Kernel of $T_{N, h} - T_{N,f} \circ T_{N,g}$. 
By the Schur test:
\begin{align}
&\norm{T_{N, h} - T_{N,f} \circ T_{N,g}}_{L^2(X, L^N ) \to L^2 (X,L^N)} ^2   \\
  & \quad \le { \left ( \sup _{x \in X} \int _X \norm{K(x,y)}_{L^2 (X,L^N)} \dd \mu (y) \right ) } {\left ( \sup _{y \in X} \int _X \norm{K(x,y)}_{L^2 (X,L^N)} \dd \mu (x) \right )}.
\end{align}
By Theorem \ref{thm.asmptotic} and \eqref{eq:asymp_better}, we can approximate the Schwartz kernels of $T_{N,f}$ and $T_{N,g}$ with $J+1$ terms. That is, we may write
\begin{align}
T_{N,f} (x,\bar y) = \underbrace{ \left ( \frac{N}{2\pi}\right ) ^d}_{:=c_d} e^{N\psi (x,\bar y) } \left ( \underbrace{ \sum _{j=0} ^{J} N^{-j} f_j (x,\bar y ) }_{:= F_{J} (x,\bar y)} \right ) + \underbrace{ e^{\frac{N}{2} (\phi(x) + \phi(y) )} \mathcal{O}(N^{- \tilde J}  m_1(x)) }_{:=R_{f,J} (x,y)}
\end{align}
where $\tilde J = (J+1)(1- 2\delta)  - d$. Define $G_{J}$ and $R_{g,J}$ similarly as an approximation of the kernel of $T_{N,g}$. Then locally:	
\begin{align}	
T_{N,f} \circ T_{N,g} (x,\bar y) = I_1 + I_2 + I_3,
\end{align}
where
\begin{align}
I_1 &: =c_d^2 \int_U e^{N (\psi(x,\bar w) - \phi(w) + \psi(w,\bar y)) } F_{J }(x,\bar w) G_{J} (w,\bar y) \dd \mu (w),\\
I_2 & := c_d \int_U \left ( e^{N\psi(x,\bar w) } F_{J} (x,\bar w) R_{g,J} (w, y) + e^{ N \psi(w,\bar y)} G_{J}(w, y) R_{f,J} (x,w) \right ) e^{-N \phi(w)}\dd \mu (w),\\
I_3 &:= \int_X e^{\frac{N}{2}(\phi(x) + \phi(y))} \mathcal{O} \left (N^{-2\tilde J } m_1(x)m_2(x) \right ) \dd \mu (w).
\end{align}
Here $U \subset X$ is a coordinate patch containing $x$ and $y$ which we assume exists, otherwise (by the same reasoning as in the proof of Theorem \ref{thm.asmptotic}), $T_{N,f}(x,\bar y)$ will be $\exp(\frac{N}{2} (\phi(x) + \phi(y)) \mathcal{O}(N^{-\infty})$ which is bounded by a constant times the $I_3$ term. Moreover, by the same reasoning, we can just integrate over $U$, as the integral over $X \setminus U$ will similarly be negligible.

$I_3$ is $ e^{\frac{N}{2} (\phi(x) + \phi(y) )}\mathcal{O} \left (N^{-2\tilde J} m_1(x) m_2(x)\right )$. Using \eqref{eq.gaussian behaviour}, $I_2$ is bounded in absolute value by:	
\begin{align}	
 cN ^{d - \tilde J + \delta M_0} e^{\frac{N}{2} (\phi(x) + \phi(y))} \int_U \left ( e^{ - Nc |x-w |^2 } |F_J (x,\bar w)| + e^{-Nc|w- y| ^2}| G_J (w,\bar y)| \right ) \dd \mu (w)  \\	
  \lesssim N^{d - \tilde J} e^{\frac{N}{2} (\phi(x) + \phi(y) )} N^{2M_0 \delta } 
\end{align}	
for some positive constant $c> 0 $. Here we used that $m_1$ and $m_2$ are bounded by a constant times $N^{\delta M_0}$.

\noindent { \textbf{Step 2:}
We will construct a $h\in S_\delta (m_1m_2)$ to cancel with the terms in $I_1$.

Define the operators $C_j : C^\infty (X; \C) \to C^\infty (X \times X ; \C)$ by $C_j(f):=  B^{-1} (x,\bar y) f_j (x,\bar y)$. Recall that $B = 1 + N^{-1 } b_ 1 + \cdots$ is the amplitude of the Bergman kernel. Note that $B$ is bounded below for $N$ sufficiently large. Explicitly, for $x,y$ contained in a sufficiently small neighborhood $U$, $C_j$ are such that	
\begin{align}	
\left ( \frac{N}{2 \pi}\right ) ^{2d} \int_U e^{N \Phi _{x,\bar y}(w)} B(x,\bar w) B(w,\bar y) f(w) \dd \mu(w) 	
\end{align}	
is
\begin{align}	
 \left ( \frac{N}{2 \pi}\right )^d e^{N\psi(x,\bar y) } B(x,\bar y) \sum _{j=0}^{J} N^{-j} C_j [f(\cdot ) ] (x,\bar y) + R_{f,J}(x,y) .	
\end{align}
where $\Phi _{x,\bar y } (w): = \psi(x,\bar w) - \phi(w) + \psi (w,\bar y).$
By \eqref{eq.fj}, $C_j$ are differential operators of order at most $2j$.

Using this notation, $I_1$ is
\begin{align}
   & c_d^2 \int_U e^{N \Phi_{x,\bar y } (w)}  B(x,\bar w ) B(w,\bar y) \left ( \sum_{k=0}^J C_k [f] (x,\bar w ) \right)  \left ( \sum_{j=0}^J C_k [g] (w,\bar y) \right) \dd \mu (w) \\
    =&c_d^2 \sum_{j=0}^{2J} N^{-j} \int_U e^{\Phi(x,\bar y)(w)} B(x,\bar w) B(w,\bar y) \left ( \sum _{a+b = j} C_a [f](x,\bar w) C_b [g](w,\bar y) \right) \dd \mu (w)\\
    =& c_d e^{N \psi (x,\bar y)} B(x,\bar y) \sum_{j=0}^{2J} N^{-j} k_j (x,\bar y) + e^{\frac{N}{2} (\phi (x) + \phi(y))}\mathcal{O} (N^{-\widetilde {2J}} m_1 (x))
\end{align}
where for each $j$
\begin{align}
k_j (x,\bar y) = \sum _{c + d =j}C_d \left [ \sum _{a+b = c} C_a[f](x,\cdot ) C_b[g] (\cdot ,\bar y) \right ] (x,\bar y). \label{eq:1267}
\end{align}
}

\noindent { \textbf{Step 3:}
If we make the ansatz that $h \sim \sum_{j=0}^{\infty} N^{-j} h_j$ for some $h_j \in  N^{2\delta j} S_{\delta } (m_1 m_2)$, then 
\begin{align}
T_{N,h}(x,\bar y)  = \left ( \frac{N}{2\pi} \right )^d e^{N\psi (x,\bar y) }  B(x,\bar y) \sum _{j=0}^{2J} N^{-j} &\sum _{c +d = j} C_d [ h_c ] (x,\bar y) \label{eq:1273}\\
& + e^{\frac{N}{2} (\phi (x) + \phi(y) )} \mathcal{O}( N^{-\widetilde{2J}} m_1(x)m_2(x)).
\end{align}
We will get cancellation of terms in $I_1$ by matching the coefficients of $N^{-j}$ of \eqref{eq:1267} with \eqref{eq:1273} which will follow if
\begin{align}
\sum _{c+d = j } C_d [h_c ] (x,\bar y ) = \sum _{c + d =j}C_d \left [ \sum _{a+b = c} C_a[f](x,\cdot ) C_b[g] (\cdot ,\bar y) \right ] (x,\bar y). \label{eq:1278}
\end{align} 
Recall that $C_0 [f(\cdot ) ] (x,\bar x) = f(x)$, therefore letting $y = x$ and rearranging \eqref{eq:1278} gives us:
\begin{align}
h_j (x) = \sum _{c + d =j}C_d \left [ \sum _{a+b = c} C_a[f](x,\cdot ) C_b[g] (\cdot ,\bar x) \right ] (x, \bar x) - \sum _{d =1}^j C_d [h_{ j - d }] (x,\bar x), \label{star constructions}
\end{align}
which inductively gives the functions $h_j$, which is sufficient as $h_0 (x) = f(x) g(x)$. This provides us with the following Lemma.

\begin{lemma}[\textbf{Derivative Estimates of $\bm{(f \star g )_j}$}]\label{lemma.derivative estimates}
There exist linear bi-differential operators $\mathcal{C}_j$ of order at most $2j$ such that $h_j (x) = \mathcal{C} _j [f,g] (x) $ agrees with $h_j$ written in \eqref{star constructions}.
\end{lemma}
\begin{proof}
This follows by induction on $j$. When $j =0 $, $h_0 = fg$ and so $\mathcal{C} _ 0 $ is a zero order operator. Next assume this is true for $j-1$ and apply \eqref{star constructions}. The first summation in \eqref{star constructions} involves derivatives of order $2j$, and by the induction hypothesis, the second summation involves derivatives of order $2j$.
\end{proof}

To construct $h$ via Borel's Theorem (Proposition \ref{prop:Borel}) we need to show that for each $j$, $h_j \in N^{2\delta j} S_\delta (m_1 m_2) $. This follows immediately by Lemma \ref{lemma.derivative estimates}, as
\begin{align}
h_j (x) = \mathcal{C}_j [f,g] (x) \in N^{2j\delta} S_\delta (m_1 m_2 ) .
\end{align}
We therefore have
\begin{align}
    T_{N,h} (x,\bar y) = I_1 + e^{\frac{N}{2} (\phi(x) + \phi(y))} \mathcal{O} (N^{-\widetilde{2J}} m_1 (x)m_2(x)).
\end{align}
}
\noindent { \textbf{Step 4:}
Putting these estimates together, we get 
\begin{align}
 K(x,y) &= T_{N,(f\star g) } - (I _1 + I_2 + I _3 ) \\
&= e^{\frac{N}{2} (\phi(x) + \phi(y))} \Big ( \mathcal{O} \Big ( N^{-\widetilde{2 J }} m_1 (x)m_2(x) \Big )\\
& \qquad  \qquad  \qquad +\mathcal{O} \Big (N^{d +2M_0 \delta-\tilde J  }\Big )+ \mathcal{O}\left (N^{-\widetilde{2J}} m_1 (x)m_2(x) \right ) \Big ) \\
&= e^{\frac{N}{2} (\phi(x) + \phi(y)) }\mathcal{O} (N^{2M_0\delta +d - \widetilde{J} }).
\end{align}
We therefore get that
\begin{align}
 \sup _{x \in X} \int _X |K(x,y) | e^{\frac{-N}{2} (\phi(x) + \phi(y))} \dd \mu (y) = \mathcal{O}(N^{-\infty}),\\
 \sup _{y \in X} \int _X |K(x,y) | e^{\frac{-N}{2} (\phi(x) + \phi(y))} \dd \mu (x) = \mathcal{O}(N^{-\infty})
\end{align}
as $J$ can be made arbitrarily large.}
\end{proof}

\section{Applications of the Exotic Calculus}\label{section:application}
This symbol calculus allows us to get a parametrix construction, a functional calculus, and a trace formula.

\subsection{Parametrix Construction}
We begin by proving a parametrix construction of Toeplitz operators associated to symbols in $S_\delta(m)$ which are elliptic with respect to $m$. This follows the usual parametrix construction for pseudo-differential operators (see for instance \cite[Theorem 4.1]{Grigis}).
\begin{theorem}[\textbf{Parametrix Construction}]\label{thm:parametrix construction}
Suppose $\delta \in [0,1/2)$, $m\ge 1$ is a $\delta$-order function on $X$ (a quantizable K\"ahler manifold), and $f\in S_\delta (m)$ is such that there exists $C>0$ and $z\in \C$ such that
\begin{align}
|f(x) - z| > C m(x)
\end{align}
for all $x\in X$. Then there exists $g\in S_\delta (m^{-1})$ such that
\begin{align}
T_{N,f-z} \circ T_{N,g} = 1 + \mathcal{O} (N^{-\infty}), && T_{N,g} \circ T_{N,f-z} = 1 + \mathcal{O}(N^{-\infty}),
\end{align} 
and the principal symbol of $g$ is $(f_0-z)^{-1}$ where $f_0$ is the principal symbol of $f$.
\end{theorem}
\begin{proof}
Define ${g_0}(x) := (f_0(x) -z )^{-1}$ so that $|{g_0}(x)| \le C m(x)$.
For each $\alpha \in \N^{2d}$ locally $\p^\alpha_{x,\bar x} {g_0} (x)$ can be estimated by the Fa\`{a} di Bruno formula:
\begin{align}
 \p^\alpha_{x,\bar x} {g_0}(x) = \sum _{\pi \in \Pi_\alpha} \frac{c_\pi}{(z-f_0)^{|\pi|+1}} \prod _{\beta \in \pi} \p^\beta_{x,\bar x}(z - f_0(x)) ,
\end{align} 
where $\Pi_\alpha$ is the set of partitions on the set $(1, 2, \dots, |\alpha|)$, $\beta \in \pi$ runs through the blocks in the partition $\pi$, and $c_\pi$ is the constant from repeatedly differentiating $x^{-1}$. For each $\pi$, note that $|\beta | = |\pi| \le |\alpha|$ for $\beta \in \pi$ so that: 
 \begin{align}
 \left | \frac{c_\pi}{(z-f_0(x))^{|\pi|+1}} \prod _{\beta \in \pi} \p^\beta_{x,\bar x}(z - f_0(x))\right | &\lesssim  (m(x))^{-|\pi| - 1} N^{|\pi| \delta} (m(x))^{|\pi|}\\
 &\lesssim (m(x))^{-1} N^{|\alpha | \delta }
 \end{align}
 {where the implied constants in the inequalities depend on $\pi$, but are independent of $N$ and $x$.}
 
We therefore have that ${g_0} \in S_\delta (m^{-1} ) $. 
Next, using Theorem \ref{Theorem:composition }, let ${s_1} \in S_\delta (1)$ be such that $T_{N,f-z}\circ T_{N,{g_0}} = 1 - N^{-(1-2\delta)} T_{N,{s_1}} + \mathcal{O} (N^{-\infty})$. Then define ${s_2}\sim \sum_{j=0}^{\infty}N^{-j(1-2\delta)} {s_1}^{\star j}$, where:
\begin{align}
{s_1}^{\star j} := \underbrace{{s_1}\star \cdots \star {s_1} }_{j \text{ terms}}.
\end{align}
By repeatedly applying Theorem \ref{Theorem:composition }, ${s_1}^{\star j} \in S_\delta(1)$ for all $j \in \Z_{\ge  0}$ so that ${s_2} \in S_\delta (1)$. Lastly, define $g := {g_0} \star {s_2}\in S_\delta(m^{-1})$. We can check that
\begin{align}
T_{N,f-z} \circ T_{N,g} &= T_{N,f-z} \circ T_{N,{g_0} \star {s_2}} = T_{N,f-z} \circ T_{N,{g_0}} \circ T_{N,{s_2}} + \mathcal{O}(N^{-\infty})\\ 
&= (1 - N^{-(1-2\delta) } T_{N,{s_1}} ) \circ  T_{N,{s_2}} + \mathcal{O}(N^{-\infty}).
\end{align}
So that for each $J\in \N$, we have 
\begin{align}
T_{N,f-z} \circ T_{N,g} &= (1 - N^{-(1-2\delta) } T_{N,{s_1}} ) \circ \left (\sum _{j=0}^J N^{-(1-2\delta)j} T_{N,{s_1}}^j \right ) + \mathcal{O}(N^{-(1-2\delta)(J+1)}) \\
&= 1 + \mathcal{O}(N^{-(1-2\delta)(J+1)}).
\end{align}
Therefore $T_{N,f-z} \circ T_{N,g} = 1 + \mathcal{O}(N^{-\infty})$ so that $g$ is a right-parametrix for $f-z$. We can similarly construct $g_\ell$ as a left-parametrix for $f-z$. But note that
\begin{align}
T_{N,g_{\ell}} = T_{N,g_{\ell}} \circ (T_{N,f-z} \circ T_{N,g} + \mathcal{O}(N^{-\infty})) = T_{N,g} + \mathcal{O}(N^{-\infty}).
\end{align}
Therefore $g$ is also a left-parametrix for $f-z$. Lastly, the principal symbol of $g$ is just the product of the principal terms of ${g_0}$ and ${s_2}$, which is just the principal term of ${g_0}$, which is $(f_0-z)^{-1}.$
\end{proof}

\subsection{Functional Calculus}
Here we present functional calculus of Toeplitz operators using the Helffer-Sj\"ostrand formula. For symbols bounded uniformly in $N$, this result is proven in \cite[Proposition 12]{Charles}. Our proof is adapted from results on functional calculus of pseudo-differential operators with symbols in similarly defined symbol classes presented by Dimassi and Sj\"ostrand in \cite[Chapter 8]{Dimassi}.

\begin{theorem}[\textbf{Functional Calculus}]\label{thm:functional calculus}
Suppose that $\delta \in [0,1/2)$, $m$ is a $\delta$-order function on $X$ (a quantizable K\"ahler manifold) such that $m\ge 1$, and $f\in S_\delta(m)$ is such that
\begin{enumerate}
\item $f(x) \in \R_{\ge 0}$ for all $x\in X$,
\item there exists $C_0> 0 $ such that $|f(x)| \ge C_0^{-1} m (x) - C_0$.
\end{enumerate}
Then for any $\chi \in C^\infty_0 (\R ; \C)$, there exists $g\in S_\delta (m^{-1})$ such that
\begin{align}
\chi(T_{N,f}) = T_{N,g} + \mathcal{O} (N^{-\infty}),
\end{align}
and the principal symbol of $g$ is $\chi(f_0)$ where $f_0$ is the principal symbol of $f$.
\end{theorem}

\begin{proof}

Let $\tilde \chi $ be an almost analytic extension of $\chi$ such that $\dbar_z \chi (z) = \mathcal{O} ( \abs{\Im z }^\infty) $. Because $f$ is real-valued, \cite[Lemma 5.1.3]{Floch} can be immediately adapted to $S_\delta (m)$ to see that $T_{N,f}$ is self-adjoint. By the Helffer-Sj\"{o}strand formula
\begin{align}
\chi(T_{N,f} ) = \frac{-1}{\pi} \int _\C \dbar _z \tilde \chi (z) (z - T_{N,f})^{-1} \dd m(z)
\end{align}
(see for instance \cite[Theorem 8.1]{Dimassi}). For $\Im z \neq 0$ we aim to construct an approximate inverse of $z - T_{N,f}$ using the parametrix construction stated in Theorem \ref{thm:parametrix construction}. But first a technical bound must be proven.
\begin{lemma}\label{claim.faa} If $s_1(z,x) = (z - f_0(x) )^{-1} $,  with $z\in \supp(\tilde \chi (z)) $ and $\Im{z} \neq 0$, then for all $\alpha \in \N^{2d}$
\begin{align}
|\p^\alpha _{x,\bar x} s_1 (z,x) |\lesssim \abs{\Im z } ^{-1 - |\alpha |} N^{ |\alpha | \delta }(m(x))^{-1}. \label{eq:943}
\end{align}
\end{lemma}
\begin{proof}
First we prove a lower bound of $|f(x) +z|$. Write $z = z_1 + iz_2$. Let $C_1 > 0$ be sufficiently large so that for $|z_1 - 1 | > C_1$ then $\tilde \chi (z) = 0$. Let $C_2 = C_0 - 1 + C_1 $ (possibly increasing $C_1$ so that $C_2 \ge 1$). We may assume that $|z_2|<C_3$ on the support of $\tilde \chi$ for some $C_3 > 1$. Then rearranging $f(x) > m(x) C_0^{-1 } - C_0$, we see that
\begin{align}
\frac{m(x)}{C_0} &< f(x)+1 - C_1 + C_2 < |f(x)+1| - |z_1-1| + C_2 \\
&< |f(x) + z_1| +C_2 < \frac{C_2 C_3 }{|z_2|} ( |f(x)+z_1| + |z_2| ) < \frac{2 C_2 C_3}{|z_2|} ( |f(x) + z_1 + i z_2|) .
\end{align}
Therefore:
\begin{align}
|f(x) +z | > \frac{1}{C_0 C_2 C_3} \abs{ \Im z} m(x) .
\end{align}
We can then apply the Fa\`{a} di Bruno formula in the same way as in the proof of Theorem \ref{thm:parametrix construction} to get that for all $\alpha \in \N^{2d}$
\begin{align}
\p^\alpha_{x,\bar x} s_1(z,x) \lesssim \abs{\Im z }^{-1-|\alpha| }(m(x))^{-1} N^{|\alpha | \delta }
\end{align}
which proves \eqref{eq:943}.
\end{proof}

We therefore have that $s_1 \in S_\delta (m^{-1} ) $, but with bounds depending on $\abs{\Im z } $. We can now apply Theorem \ref{thm:parametrix construction} to construct $s_2 (z,x) \in S_\delta (m^{-1})$ such that
\begin{align}
T_{N,z-f} \circ T_{N,s_2} + \mathcal{O}(N^{-\infty}) = T_{N,s_2} \circ T_{N,z-f} + \mathcal{O}(N^{-\infty}) = 1.
\end{align}

For any $J\in \N$, we can approximate $s_2$ by a finite sum of elements of $S_\delta (m^{-1})$ (denoted by $s_3$) such that
\begin{align}
T_{N,z-f} \circ  T_{N,s_3}  = 1 + \mathcal{O}(N^{-J}).
\end{align}
For such a symbol $s_3$, for any $\alpha \in \N^{2d}$, there exists $C = C(J,\alpha) > 0$ such that:
\begin{align}
|\p^\alpha _{x,\bar x} s_3(z,x) | \lesssim  N^{|\alpha | \delta} \abs{\Im z}^{-C} (m(x))^{-1}. \label{eq:1302}
\end{align}
%
%
%

Therefore:
\begin{align}
s_4 (x) : = \frac{-1}{\pi} \int _\C \dbar _z \tilde \chi (z) s_3 (z,x) \dd m(z)
\end{align}
exists for all $x$ because $\dbar _z \tilde \chi = \mathcal{O} (\abs{\Im z }^\infty )$. By differentiating $s_4$, applying \eqref{eq:1302}, and using $\dbar _z \tilde \chi = \mathcal{O} (\abs{\Im z }^\infty )$, we also see that $s_4 \in S_\delta (m^{-1})$. We finally check that $T_{N, s_4}$ is an approximation of $\chi (T_{N,f} )$. Suppose $u\in H^0(X,L^N)$, then for $x\in X$
\begin{align}
T_{N, s_4} [u](x) &= \Pi_N \left [ \frac{-1}{\pi} \int _\C \dbar _z \tilde \chi (z) s_3 (z,x) u(x) \dd m(z) \right ] \\
&= \frac{-1}{\pi} \int_\C \dbar_z \tilde \chi (z) \Pi_N [s_3(z,x) u (x) ] \dd m (z) \\
&= \left ( \frac{-1}{\pi} \int_\C \dbar_z \tilde \chi (z) T_{N,s_3} \dd m (z) \right ) [u] \\
&= \left ( \frac{-1}{\pi} \int_\C \dbar_z \tilde \chi (z) ( (z - T_{N,f} )^{-1} +\mathcal{O}(N^{-J}) )\dd m (z) \right ) [u] .
\end{align}
Therefore $T_{N,s_4} = \chi (T_{N,f} ) + \mathcal{O} (N^{-J})$. Since $J$ was arbitrary, by Borel's theorem, there exists $g\in S_\delta (m^{-1} ) $ such that
\begin{align}
 \chi (T_{N,f}) = T_{N,g} + \mathcal{O} (N^{-\infty}).
 \end{align} 
 The principal symbol can be easily computed. Unraveling the above, the principal symbol of $s_3$ is $(z - f_0 (x) )^{-1}$, so that the principal symbol of $g$ is:
 \begin{align}
 \frac{-1}{\pi} \int_\C \dbar _z \tilde \chi (z) (z - f_0 (x) )^{-1 } \dd m(z) = \chi(f_0 )
 \end{align}
by the Cauchy integral formula. \end{proof}

This can be generalized for Toeplitz operators with a negligible term.
\begin{theorem}
Suppose $\delta,m,f$ satisfy the hypotheses of Theorem \ref{thm:functional calculus}, and $\set{R_N}_{N\in \N}$ is a family of linear operators mapping $H^0(X,L^N)\to H^0 (X,L^N)$ such that $\norm{R_N} = \mathcal{O} (N^{-\infty})$ and ${T_{N,f}} + R_N$ are self-adjoint for all $N$. Then for any $\chi \in C_0^\infty (\R ; \C)$, there exists $g\in S_\delta (m^{-1})$ such that:
\begin{align}
\chi (T_{N,f} + R_N ) = T_{N,g} + \mathcal{O}(N^{-\infty})
\end{align}
and the principal symbol of $g$ is $\chi(f_0) $ where $f_0$ is the principal symbol of $f$.
\end{theorem}

\begin{proof}
Let $\tilde \chi$ be an almost analytic extension of $\chi$, so that
\begin{align}
\chi (T_{N,f} + R_N ) = \frac{-1}{\pi} \int_\C \dbar _z \tilde \chi(z) (z - T_{N,f} - R_N) ^{- 1} \dd m(z).
\end{align}
But note that
\begin{align}
(z - T_{N,f} )^{-1 } - (z - T_{N,f} - R_N) ^{-1} = (z - T_{N,f})^{-1} R_N (z - T_{N,f} - R_N)^{-1}.\label{eq:968}
\end{align}
Both $ (z - T_{N,f})^{-1} $ and $ (z - T_{N,f} - R_N)^{-1}$ have operator norm controlled by $N$ to some finite power, so that the right-hand side of \eqref{eq:968} is $\mathcal{O}(N^{-\infty})$. Therefore:
\begin{align}
\chi (T_{N,f} + R_N ) = \frac{-1}{\pi} \int_\C \dbar _z \tilde \chi(z) (z - T_{N,f} ) ^{- 1} \dd m(z) + \mathcal{O}(N^{-\infty})
\end{align}
and we just follow the rest of the proof of Theorem \ref{thm:functional calculus}.
\end{proof}

\subsection{Trace Formula}
A critical result required in proving a probabilistic Weyl-law for Toeplitz operators in \cite{Oltman} is a trace formula. Fortunately, this is straightforward to compute by the explicit kernel expansion described in Theorem \ref{thm.asmptotic}.

\begin{theorem}[\textbf{Trace Formula}]\label{thm:trace formula}
For $\delta \in [0,1/2)$, suppose $m$ is a $\delta$-order function on $X$ (a quantizable K\"ahler manifold). Then if $f\in S_\delta (m)$,
\begin{align}
\Tr( T_{N,f}) = \left ( \frac{N}{2\pi} \right )^d \int _X f_0(x) \dd\mu(x) + \left ( \int_X m(x) \dd\mu (x) \right ) \mathcal{O}(N^{d-(1-2\delta)}) .\label{eq:1418}
\end{align}
\end{theorem}
\begin{proof}
By \cite[Proposition 6.3.4]{Floch} and Theorem \ref{thm.asmptotic} (specifically \eqref{eq:asymp_better} with $J=1$), 
\begin{align}
\Tr (T_{N,f} ) &= \int_X T_{N,f}(x,\bar x) e^{-N\phi(x)} \dd\mu (x)\\
&= \left (\frac{N}{2\pi}\right )^d \int_X \left ( f_0(x,\bar x) + \mathcal{O}(N^{-(1-2\delta)} m(x) ) \right ) \dd \mu(x).
\end{align}
By Theorem \ref{thm.asmptotic}, $f_0(x,\bar x)$ is $f_0(x) + \mathcal{O}( N^{-(1-2\delta)}m(x))$ and so \eqref{eq:1418} follows.
\end{proof}

\appendix

\section{Computation of the second term in the star product}\label{Appendix:1}

The goal of this section is to compute the second term in the star product of Toeplitz operators. Indeed, by Theorem \ref{Theorem:composition } we know that if $f$ and $g$ are symbols in $S_\delta (m)$, then there exists a symbol $h\sim \sum N^{-j} h_j$, such that $T_{N,f} \circ T_{N,g} = T_{N,h} + \mathcal{O} (N^{-\infty})$. It is straightforward to show that $h_0 = fg$ (modulo $\mathcal{O}(N^{-(1-2\delta)})$). This section directly computes $h_1$ (modulo $\mathcal{O} (N^{-2(1-2\delta)}m)$ error).

In this section, for vectors $u,v\in \C^d$, we write $\ip{u}{v} := \sum u_i v_i$. For functions $f\in C^\infty(\C^d)$, we denote by $\nabla_x f$ the vector in $\C^d$ whose $j^{th}$ component is $\partial_{x_j} f $. We similarly denote by $\nabla _{\bar x} f$ the vector whose $j^{th}$ component is $\dbar _{x_j} f$.

\begin{theorem}\label{thm:second_term}
Given $\delta \in [0,1/2)$, suppose $m_1,m_2$ are two $\delta$-order functions on $X$ (a quantizable K\"ahler manifold with K\"ahler potential $\phi$), $f\in S_\delta (m_1)$, $g\in S_\delta (m_2$), and $h = f \star g$. Then locally
{
\begin{equation}\label{eq:apend}\begin{split}
h(x) = f(x)g(x)  - {N^{-1}} \sum _{j,k =1}^d ( \p \dbar \phi (x))^{j,k}& \p_k f (x) \dbar _j g (x) \\
&+ \mathcal{O} (N^{-2(1-2\delta)} m_1(x)m_2(x) ) , \end{split}
\end{equation}}
where $ ( \p \dbar \phi (x))^{j,k} \in \C^{d\times d}$ is such that $\sum _k( \p \dbar \phi (x))^{j,k} (\p_k \dbar _\ell \phi(x) ) = \delta _{j,\ell} $.
\end{theorem}

\begin{rem}
From this, we get the classical-quantum correspondence of Toeplitz operators. Indeed, by \eqref{eq:apend} the principal symbol of $[T_{N,f} , T_{N, g } ] $ is
\begin{align}
N ^{-1} ( -\ip{(\p \dbar \phi)^{-1} \nabla_x f }{\nabla _{\bar x}  g} + \ip{(\p \dbar \phi)^{-1} \nabla _x g }{\nabla _{\bar x } f}). \label{eq:anti}
\end{align}
Note that the Poisson bracket of $f$ and $g$ is $\{ f , g \} = \omega (X_f , X_g)$, where $X_f$ and $X_g$ are the Hamiltonian vector fields of $f$ and $g$ respectively and $\omega$ is the symplectic form on $X$. If we write $\omega = \sum W_{i,j} dz_i\land d\bar z_j$, then
\begin{align}
\dbar_{x_i} f = \sum_{j=1}^{d} W_{j,i} dz_j (X_f) ,
 && \p_{x_i} f = - \sum_{j=1}^{d} W_{i,j} d\bar z_j (X_f),
\end{align}
with identical identities relating $g$ and $X_g$. Therefore
\begin{align}
\omega(X_f , X_g) &= \sum _{i,j} W_{i,j} dz_i \land d\bar z _j (X_f , X_g) \\
&= \sum _{i,j} W_{i,j} ( dz_i (X_f) d\bar z_j (X_g) - dz_i (X_g) d\bar z_j (X_f)) \\
&= \ip{W^t (W^t)^{-1} \nabla _{\bar x} f }{- W^{-1} \nabla_x g} - \ip{W^t (W^t)^{-1} \nabla _{\bar x} g }{-W^{-1} \nabla_x f} \\
&= \ip{W^{-1} \nabla _x f}{\nabla _{\bar x} g} - \ip{W^{-1} \nabla _x g }{\nabla _x f}.
\end{align}
Now, because $W_{i,j} = i \p_i \dbar _j \phi $, we see from \eqref{eq:anti} that
\begin{align}
[T_{N,f}, T_{N,g} ] = \frac{1}{ iN} T_{N,\{ f , g \}} + \mathcal{O} (N^{-2(1-2\delta) } m ) .
\end{align}
\end{rem}

The method to prove Theorem \ref{thm:second_term} is to compute the Schwartz kernel of the asymptotic expansion of $T_{N,f} \circ T_{N,g}$ and find a symbol that agrees with this kernel. By the almost analytic properties of the kernel, it suffices to work exclusively on the diagonal. Along the diagonal, the method of stationary phase has more explicit formulae. This section will use a stationary phase expansion presented in \cite{Hormander}.

\begin{proof}
Estimates on the error term in \eqref{eq:apend} were established in Theorem \ref{thm.asmptotic}. For a simpler proof, we assume that $f,g\in S_0(1)$.

Near $x_0 \in X$, we choose a normal coordinate system $(z^1 (x) ,\dots, z^d(x)) \in \C^d$. In this way, $\p_{z_j} \p_{\bar z_k} \phi(z(x_0) ) = \delta _{j,k}$ and $\p_{z,\bar z}^\alpha \p_{z_j} \p_{\bar z_k} \phi(z(x_0)) = 0$ for all $j,k = 1,\dots , d$ and $\alpha\in \N^{2d}$ with $|\alpha | = 1$.

Let $C_j$ be the differential operators coming from stationary phase:
\begin{align}
\left ( \frac{N}{2\pi} \right )^{2d} \int_{\C^d} u(w) e^{N \Phi_{x\bar y} (w) }\dd \mu (w) \sim \left ( \frac{N}{2 \pi} \right )^d e^{N \psi(x,\bar y)} \sum _0^\infty N^{-j} C_j [u](x,\bar y),\label{eq.24}
\end{align}
with $u\in C^\infty (\C^d ; \C)$, $\Phi _{x,\bar y} (w) = \psi(x,\bar w) - \phi(w) +\psi(w,\bar y)$, and $\mu (w) = \omega^{\land d} / d!$. When computing $T_{N,f}$, terms of order $\mathcal{O}(N^{-2})$ are not needed to compute the second term in the expansion. In this case the functions coming from the Bergman kernel expansion (recall \S \ref{subsection:Bergman kernels}) can be approximated as $B(x,\bar y) = 1 + N^{-1} b_1 (x,\bar y) + \mathcal{O} (N^{-2})$, so that the amplitude in the kernel of $T_{N,f}$ is $f(w) (1 + N^{-1} (b_1 (x,\bar w) +b_1 (w,\bar y) ) + \mathcal{O} (N^{-2} ) $. In this way:
\begin{align}
f_0 (x,\bar y) &= C_0[f](x,\bar y), \\
f_1 (x,\bar y) &= C_0[f(\cdot ) (b_1(x,\cdot ) + b_1(\cdot , \bar y) ] (x,\bar y) + C_1 [f](x,\bar y),
\end{align}
and on the diagonal:
\begin{align}
f_0 (x,\bar x) &= f(x),\\
f_1 (x,\bar x) &= 2f(x)b_1(x)+ C_1 [f](x,\bar x). \label{eq.identity}
\end{align}
If we are given $T_{N,f}$ and $T_{N,g}$, then the first term in the expansion of $T_{N,f} \circ T_{N,g}$ along the diagonal will be $f(x) g(x)$. While the second term is
\begin{align}
C_0 [f_1 (x,\cdot ) g_0 (\cdot, \bar y) + f_0 (x, \cdot ) g_1 (\cdot ,\bar y ) ] (x,\bar y) + C_1 [f_0 (x,\cdot )g_0 (\cdot, \bar y) ] (x,\bar y). 
\end{align}
Along the diagonal, this is
\begin{align}
(2fb_1 + C_1 [f] ) g + f(2gb_1 + C_1 [g]) + C_1 [f_0 (x,\cdot ) g_0 (\cdot , \bar x ) ], \label{s1}
\end{align}
with all $C_j$ operators evaluated at $(x,\bar x)$ and functions evaluated at $x$. Suppose $T_{N,f} \circ T_{N,g} = T_{N,h} + \mathcal{O} (N^{-\infty})$ for some $ h\sim\sum N^{-j} h_j$. The $N^0$ order term of $T_{N,h} (x,\bar x)$ is $C_0 (h_0)(x) = h_0 (x)$, so that $h_0(x) = f(x) g(x)$. The $N^{-1}$ order term is
\begin{align}
C_0 [ h_0(b_1 (x,\cdot ) + b_1 (\cdot ,\bar y) ) + h_1 ] (x,\bar y)+ C_1 [h_0](x,\bar y).
\end{align}
Along the diagonal this is:
\begin{align}
2h_0(x) b_1(x) + h_1(x) +C_1 [fg](x,\bar x). \label{eq:1114}
\end{align}
Setting \eqref{eq:1114} equal to \eqref{s1}, and solving for $h_1$ gives the relation
\begin{align}
h_1 (x)& = 4fgb_1 -2fgb_1-C_1 [fg] + C_1 [f_0 (x,\cdot ) g_0 (\cdot ,\bar x)] + C_1 [f] g + f C_1 [g]\\
&= 2fg b_1 + C_1 [f_0 (x,\cdot ) g_0 (\cdot , \bar x) -f(\cdot )g(\cdot ) ] + g C_1 [f] +fC_1 [g] \label{eq.final}
\end{align}
with all $C_j$ operators evaluated at $(x,\bar x)$, $f,g$ evaluated at $x$, and $b_1$ evaluated at $(x,\bar x)$.

Recall that $i\p \dbar \phi = \omega$ and in normal coordinates $\omega(x) = iH $ with $H$ a positive definite, real, self-adjoint, invertible matrix, such that $H(x_0 ) = 1$.
\begin{lemma}
On the diagonal:
\begin{align}
C_1 [u] = L_1 [u \det H] 
\end{align}
with $L_1 = -\ip{ \nabla _z}{\nabla _{\bar z }}+ A$, and
\begin{align} 
A = -2^{-1}\ip{ \nabla _z }{\nabla _{\bar z}} \det H(x_0) - b_1(x_0)
\end{align}
\end{lemma}
\begin{proof}
By \cite[Theorem 7.7.5]{Hormander}
\begin{align}
\int_{\R^{2d}} u(w) e^{N \Phi _{x,\bar x} (w)} \dd w \sim \dfrac{e^{N \Phi _{x,\bar x} (x)} (2\pi)^{d} }{\sqrt{\det ( - N \Phi '' _{x \bar x} (x))}} \sum _{j=0}^\infty N^{-j} L_j u, \label{eq:1156}
\end{align}
with:
\begin{align}
L_1 u = \sum_{\nu = 1}^3 i ^{-1 + \nu } 2^{-\nu}\ip{(\Phi ''_{x, \bar x})^{-1} (x) D}{D}^\nu (\mathfrak{g}_x^{\nu-1} u ) (w) / ((\nu-1)! \nu! ) , \label{eq.L1}
\end{align}
with derivatives evaluated at $x$, and
\begin{align}
\mathfrak{g}_x (w) = \Phi_{x, \bar x} (w) / i - \Phi_{x, \bar x} (x) /i + \ip{i \Phi_{x, \bar x} (x) '' (w-x) }{w-x} /2 .
\end{align}

By computation
\begin{align}
\Phi''_{x_0,\bar x_0} (x_0) = -4 \mat{ 1 & 0 \\ 0 & 1},
\end{align}
so that
\begin{align}
\det (- N \Phi''_{x_0 ,\bar x_0} (x_0) )^{-1/2} = (4N)^{-d}.
\end{align}

Therefore, \eqref{eq:1156} simplifies to:
\begin{align}
\left( \frac{\pi}{2N}\right )^d \sum _{j=0}^\infty N^{-j} L_j (u) \label{eq:1213}
\end{align}
Observe that $\ip{\Phi '' _{x_0 ,x_0} (x_0) ^{-1} D}{D} = 4^{-1} \Delta_w$ (using the notation that $D = i^{-1}\nabla$). Let $\mathfrak{g} := \mathfrak{g}_{x_0}$, and note that $\mathfrak{g}$ vanishes to third order at $x_0$. Then we compute that
\begin{align}
L_1 u = \frac{1}{8} \Delta _w u + c_2 ((\nabla_w)^2 \mathfrak{g})u + c_3 (\nabla _w (\Delta _w \mathfrak{g}) )\cdot \nabla _wu + c_4 ((\Delta_w)^3 (\mathfrak{g}^2) u) \label{eq:nasty}
\end{align}
for some constants $c_2,c_3,c_4$. Observe that $(\nabla _w (\Delta _w \mathfrak{g}) )$ evaluated at $x_0$ will be a linear combination of first derivatives of the entries of $H(x)$, which are all zero because we are using normal coordinates. Therefore \eqref{eq:nasty}, evaluated at $x_0$, reduces to:
\begin{align}
L_1 u &= \left ( \frac{1}{8} \Delta _w+ A \right )u = \left ( \frac{1}{2} \nabla _z \cdot \nabla_{\bar z} + A \right )u \label{eq:1222}
\end{align}
for some constant $A$, and using the complex variable $z = w_1 +iw_2$.

The operators $C_j$ along the diagonal can be recovered from $L_j$. Indeed, by matching powers of $N$ and using that $\dd \mu =\omega^\land d / d!= \det (H) 2^d \dd m(w)$ (see \cite[Lemma 2.6.2]{Floch}), we see, by \eqref{eq:1213},
\begin{align}
C_j [u] & = L_j [u \det H]. \label{eq:C_j L_j}
\end{align}

The constant $A$ can be computed by recalling that if $f=1$, the Toeplitz operator is just the Bergman projector. So letting $f = 1$ in \eqref{eq.identity}, we get that $C_0 [2b_1 ] + C_1 [1] = b_1$, which can be rearranged as $b_1 = -C_1 (1)$. Then since $\det H(x_0) = 1$, and using \eqref{eq:C_j L_j} and \eqref{eq:1222}, we get
\begin{align}
b_1 = - A - 2^{-1}\ip{\nabla _z}{\nabla _{\bar z }} \det H(z). \label{eq:A def}
\end{align}

\end{proof}

Because we are using normal coordinates, $\nabla \det H = 0$, therefore using \eqref{eq:A def},
\begin{align}
C_1[u] &= ( 2^{-1}\ip{ \nabla _z}{\nabla _{\bar z }} + A ) (u \det H ) \\
&= Au + 2^{-1} \ip{ \nabla _z}{\nabla _{\bar z }} u + 2^{-1} (\ip{ \nabla _z}{\nabla _{\bar z }} \det (H) ) u \\
&= 2^{-1}\ip{\nabla _z}{\nabla _{\bar z }} u - b_1 u.
\end{align}
Then \eqref{eq.final} becomes, after canceling all $fgb_1$'s,
\begin{align}
2^{-1 } (\ip{ \nabla _z}{\nabla _{\bar z }} (f_0 (x,\cdot ) g_0(\cdot ,\bar x) - f(\cdot )g(\cdot )) - g \ip{ \nabla _z}{\nabla _{\bar z }}f -f \ip{\nabla _z}{\nabla _{\bar z }} g )\label{God}
\end{align}
Now note that $f_0$ and $g_0$ are almost holomorphic in the first argument, and almost anti-holomorphic in the second coordinate. They can be treated as holomorphic and anti-holomorphic as we are on the diagonal. So \eqref{God} becomes, after applying the product rule and canceling terms:
\begin{align}
-2^{-1}\ip{\nabla_z f}{\nabla_{\bar z } g}. \label{eq:last}
\end{align}
Finally, if we use arbitrary holomorphic coordinates $x$, and let $J = Dx/Dz$ be the Jacobian relating the $x$ coordinates to the normal coordinates, then \eqref{eq:last} is
\begin{align}
- 2^{-1}\ip{J^t \nabla _x f }{ (\bar J )^t\nabla _{\bar x}g} = - 2^{-1}\ip{\bar J J^t \nabla_x f}{\nabla _{\bar x}g }. \label{eq:fuq}
\end{align}
Because we used normal coordinates, $J$ must satisfy
\begin{align}
2I = J^t (\p_x \dbar _x \phi) \bar J = J^t(H )\bar J,
\end{align}
so that $J^T = 2\bar J^{-1} H^{-1}$, so that \eqref{eq:fuq} becomes:
\begin{align}
-\ip{H^{-1} \nabla _ x f }{\nabla_{\bar x } g}.
\end{align}
Then, because $H = \p \dbar \phi$, we get our theorem. \end{proof}

\section{Simple Example Worked Out}\label{section:AppendixB}
A simple (although not compact) K\"{a}hler manifold is $\C$ with symplectic form $\omega = i dz \land d \bar z$. 
Considering holomorphic sections of powers of the trivial line bundle, the quantum space for each $N$ can be identified with all holomorphic functions $f$ such that 
\begin{align}
\int_\C |f|^2 e^{-N|z|^2} dm(z) < \infty.
\end{align}
This space is called the Bargmann space with $L^2$ structure
\begin{align}
\ip{f}{g} : =2 \int_\C f (z) \bar g(z) e^{-N|z|^2 } \dd m(z).
\end{align}
{There is extensive literature on this construction, see for instance \cite[\S 1.6]{folland1989harmonic} and \cite[Chapter 13]{Zworski} among others.}

In this case, the Bergman kernel is
\begin{align}
\Pi_N (x,\bar y) = \frac{N}{2\pi} \exp( Nx \bar y ).
\end{align}
The K\"ahler potential is $\phi(y) = |y|^2 $ with analytic extension $\psi(x,\bar y) = x \bar y $ (see for example \cite[Example 7.2.2]{Floch}). So if $f\in S_\delta(1)$ for a fixed $\delta \in [0,1/2)$ , the kernel of the Toeplitz operator $T_{N,f}$ is
\begin{align}
T_{N,f} (x,\bar y) = \left ( \frac{N}{2\pi}\right )^2 \int_\C f(w) \exp(N ( x\bar w - |w|^2 + w\bar y  ) )2 \dd m(w) .
\end{align}
We write this as an integral over $\R^2$ by letting $w = w_1 + i w_2$. Completing the square of the phase, this integral is:
\begin{align}
\left( \frac{N}{2\pi} \right )^2 e^{Nx \bar y  }\int_{\R^2} e^{N( -(w_1 - a)^2 -(w_2 - b)^2)} f(w_1+iw_2) 2 \dd w_1\dd w_2\label{eq:1624}
\end{align}
with $a = 2^{-1}(x+ \bar y) $ and $b = (2i )^{-1} (x - \bar y)$, which is approximately true for quantizable K\"ahler manifolds\footnote{In the general case, this $(a,b)$ is ${p_{\rm{crit}}} (t)$. Much of the trouble with the method of complex stationary phase is that the phase is not holomorphic, and therefore the extension is not unique, and so the critical point is no longer unique. However, when the phase is not holomorphic, the critical point still approximately takes this form.}. Note by the Gaussian decay in the integrand of \eqref{eq:1624} it suffices to assume $f$ is compactly supported as anything away from $a$ or $b$ will be exponentially small in $N$.

 We may now integrate \eqref{eq:1624} as an iterated integral. Let's first integrate over $w_1$. For $R$ sufficiently large, let $a = a_1 + i a_2$, and rewrite the inner integral in \eqref{eq:1624} as 
\begin{align}
\int _{-R}^R e ^{-N (w_1 -ia_2 )^2} f(w_1+ a_1 + i w_2) \dd w_1. \label{eq:1138}
\end{align}
Let $f_\R (x,y) = f(x+ i y)$, so the integrand has the term $f_\R(w_1 + a_1 , w_2)$. By Stokes' Theorem, \eqref{eq:1138} is
\begin{align}
\int_{-R}^R e^{-N w_1^2 } \tilde f_\R (w_1 + i a_2 + a_1 , w_2) \dd w_1 + \iint_{\Omega_a} e^{-N( z - ia_2)^2} \dbar _z \tilde f_\R(z+ a_1 ,w_2) \dd z\land \dd \bar z 
\end{align}
where $\tilde f_\R $ is an extension of $f_\R$ to $\C^2$, and $\Omega_a = \set{x + iy : x\in [-R,R] , y\in [0,a_2]}$. Ignoring the second term for the moment, we now integrate the first term over $w_2$, by the same reasoning (possibly increasing $R$), we get
\begin{align} \label{eq:1145}
\int_{-R}^R \int_{-R}^R e ^{-N(w_1^2 + w_2^2)} & \tilde f_\R(w_1 + a , w_2 + b) \dd w_1 \dd w_2\\
& + \int_{-R}^R \iint_{\Omega_b} e^{-N(w_1^2 + (z - ib_2)) } \dbar _z \tilde f_\R(w_1 + a , z+b) (\dd z\land \dd \bar z) \dd w_1. 
\end{align}
The first term in \eqref{eq:1145} is estimated using the method of steepest descent (for example see \cite[Exercise 2.4]{Grigis}), as
\begin{align}
\left ( \frac{\pi}{N} \right ) \sum _{k=0}^{M-1} \frac{N^{-k}}{ 4^kk!} \Delta^k \tilde f_\R (a ,b) + S_M (f ,N),
\end{align}
with
\begin{align}
|S_M (f,N) | \le C_{N} N^{-M - 1} \sum _{|\alpha | = 2M} \sup | \p^\alpha \tilde f_\R |.
\end{align}

Here $\Delta \tilde f_\R(x,y):= (\p_{\Re x } ^2 + \p_{\Re y}^2)(\tilde f_\R(x,y))$. If we compute the kernel on the diagonal, $x = y$, then all derivatives are tangential to the totally real submanifold which $\tilde f_\R$ is extended from and we evaluate the derivatives at $(\Re x, \Im x )$. So when $x = y$, the first term in \eqref{eq:1145} is
\begin{align}
&\left ( \frac{\pi}{2N}\right ) \sum _{k=0}^M \frac{N^{-k}}{ 4^k k!} (\p^2_{u} + \p^2_v )^k (f(u+ i v) )\Big|_{\substack{u = \Re{x} \\ v = \Im{x}}} + S_M(f,N) \label{eq:1163}\\
= &\left ( \frac{\pi}{2N}\right ) \sum _{k=0}^M \frac{N^{-k}}{ k!} (\p \dbar)^k f(x) + S_M(f,N).
\end{align}

\subsection{Controlling Error Terms}

Next we show that the error terms
\begin{align}
I_1 &: = \left ( \frac{N}{2\pi} \right )^{2} e^{N x\bar y } \int _\R e^{-N(w_2 -b)^2} \iint_{\Omega_a} e^{-N (z - ia_2) ^2} \dbar _z \tilde f_\R(z+a_1,w_2 ) \dd z \land \dd \bar z \dd w_2, \\
I_2 & := \left ( \frac{N}{2\pi}\right )^2 e^{N x \bar y } \int_\R \iint _{\Omega_b} e^{-N (w_1^2 + (z - ib_2 ))}\dbar _z \tilde f_\R(w_1 + a, z + b) \dd z\land \dd \bar z  \dd w_1,
\end{align}
are $\exp(-\frac{N}{2} (|x|^2 + |y|^2))\mathcal{O}(N^{-\infty})$. First note that $2\Re {x \bar y  } = -|x -y|^2 + |x|^2 + |y|^2$. Let $\e = x - y$, so that $a = \Re x - \bar \e /2$, $b = \Im x - \bar \e / (2i)$, $a_2 = \Im \e/2 $, and $b_2 = \Re \e /2$. Therefore, $|I_1|$ is bounded by
	\begin{align}	
\left ( \frac{N}{2\pi}\right )^2 e^{ \frac{N}{2} (|x|^2 + |y|^2) } \Big| \int_\R e^{-N(w_2 - b)^2 }\int _{-R}^R \int _{0}^{ \frac{\Im \e}{2}} e^{-N(z_1+ iz_2 - i \Im \e /2 )^2 - N |\e|^2/2)}  \\	
\cdot \dbar _z \tilde f _\R (z+ a_1 , w_2 ) \dd z_1 \dd z_2 \dd w_2 \Big|.	
\end{align}
We then apply $\dbar$ estimates for $\tilde f_\R$. That is for each $M \in \N$, there exists $C_M > 0$ so that $\dbar _z\tilde f_\R (a,b) \le C_M N^{\delta M_0} (\abs{\Im a } + \abs{\Im b })^M $. Fixing, $M$, the inner integral is bounded by
\begin{align}
C_M \int_0^{\frac{\Im \e}{2}} e^{N(z_2 - \Im \e / 2) ^2 -N |\e|^2/2 } N^{\delta M_0 }z_2^M \dd z_2 \label{eq:1183}
\end{align}
Expanding the exponential, we see that
\begin{align}
\eqref{eq:1183} & \le C_MN^{\delta M_0} e^{-\frac{N}{4} \Im \e ^2} \int_0^{\frac{\Im \e}{2}} \exp\left ( Nz_2^2 - z \Im\e N \right ) \dd z_2\\
 &\le C_MN^{\delta M_0} e^{-\frac{N}{4} \Im \e ^2} \frac{\Im \e ^{M+1}}{N^{M+1}} \int_0^{\frac{N}{2}} t^M \exp \left ( t \Im \e ^2 \left ( \frac{t}{N} - 1 \right ) \right ) \dd t \\
&\le C_MN^{\delta M_0} e^{-\frac{N}{4} \Im \e ^2} \frac{\Im \e ^{M+1}}{N^{M+1}} \int_0^{\frac{N}{2}} t^M e^{-t\Im \e^2 / 2} \dd t\\
&\lesssim _M e^{-\frac{N}{4} \Im \e ^2} \frac{ N^{\delta M_0}}{N^{M+1} \Im{\e} ^{M+1}} \int_0^{\frac{\Im \e^2 N}{4}} e^{-t} t^M \dd t\\
&\lesssim_M N^{\delta M_0 - M - 1}.
\end{align}
Therefore:
\begin{align}
|I_1| &\lesssim _M N^{\delta M_0 - M - 1 +2} e^{\frac{N}{2} (|x|^2 + |y|^2) } \int_\R e^{-N (w_2 - b)^2 } \int_{-R}^R e^{-Nz_1^2 } \dd z_1 \dd w_2 \\
&\lesssim N^{\delta M_0 - M +1} e^{\frac{N}{2} (|x|^2 + |y|^2) },
\end{align}
so that $I_1 = e^{\frac{N}{2} (|x|^2 + |y|^2)} \mathcal{O} (N^{-\infty})$. An identical argument is used to show the same bound for $I_2$.

\subsection{Using the notation presented in Tr\`eves}\label{appendix:treves}

It is possibly instructive to see how the change of variables presented by Tr\`eves in \cite{Treves}, and used in Section \ref{section:composition}, applies to this simple example. Let's consider a symbol $f$ to quantize. Then, as in \eqref{one part rewritten},
\begin{align}
T_{N,f}(x,\bar y) &= \left ( \frac{N}{2 \pi} \right ) e^{\frac{N}{2} (|x|^2 + |y|^2 ) } \int_{\R^{2}} e^{N \Psi (p,t) } g(p,t) \dd p ,
\end{align}
with $x = (t_1 + i t_2), y = (t_3 + i t_4)$, $w = p_1 + i p_2$, and
\begin{align}
\Psi(p,t) &:= x\bar w - |w|^2 + w \bar y -\frac{1}{2} (  |x|^2 +|y|^2 ) \\
&= x\bar w - |w|^2 + w \bar y - \frac{1}{2} (|x|^2+|y|^2)\\
&= (t_1 + i t_2 ) (p_1 - i p_2) - p_1^2 - p_2^2 + (p_1 + i p_2 )(t_3 - i t_4) - \frac{1}{2} (t_1 ^2 + t_2^2 +t_3^2 +t_4^2),\\
g(p,t) &:= 2f(p_1 + i p_2).
\end{align}
The critical point for this phase when it is holomorphically extended (note $\Psi$ is real-analytic, so the extension is unique) is
\begin{align}
{p_{\rm{crit}}}(t) &= \left ( \frac{1}{2} (t_1 + i t_2 + t_3 - it_4 ) , \frac{1}{2i} (t_1 + i t_2 - t_3 + i t_4 ) \right ),
\end{align}
so that
\begin{align}
\tilde \Psi_{pp}({p_{\rm{crit}}}(t)) &= -2 \mat{1 & 0 \\0 & 1}. 
\end{align} 
In this case, we change variables, as in Lemma \ref{lemma.q_lemma},
\begin{align}
q(p,t) &= \sqrt{2} \mat{1 & 0 \\ 0 & 1} (p - {p_{\rm{crit}}}(t) )\\
&= \sqrt{2} \mat{p_1 - \frac{1}{2} (t_1 + i t_2 + t_3 - it _4 ) \\ p_2 - \frac{1}{2i} (t_1 + it_2 - t_3 + it_4 )} : \C^2 \times \R^4 \to \C^2.
\end{align}
Then the new contour is $\set{p \in \C^2 : q(p,t) = w\in \R^2}$. For each $w \in \R^2$, we see that
\begin{align}
w_1 &= \sqrt{2} \Re{ p_1} - \frac{1}{\sqrt{2}} (t_1 + t_3) , &&& w_2 & = \sqrt{2} \Re{ p_2} - \frac{1}{\sqrt{2}} (t_2 + t_4) ,\\
0 & = \Im { p_1} + \frac{1}{\sqrt{2}}(-t_2 + t_4), &&& 0 &= \Im {p_2} + \frac{1}{\sqrt{2}} (t_1 - t_3).
\end{align}
Therefore the new contour is 
\begin{align}
U_0 &= \set{ p(w) : = \mat{ \frac{1}{2} (t_1 + t_3) + \frac{w_1}{\sqrt{2}} - \frac{i}{2} (-t_2 + t_4) \\ \frac{1}{2} (t_2 +t_4) + \frac{w_2}{\sqrt{2}} -\frac{i}{2} (t_1 -t_3)} : (w_1 ,w_2 ) \in \R^2 }.
\end{align}
The real stationary phase is applied to the amplitude $g(p(w) ) \det (\frac{\p p}{\p w})$, which is, after replacing $t$ with its definition,
\begin{align}
 \frac{1}{2} \tilde g\left (\frac{1}{2 } (x+ \bar y) + \frac{w_1}{\sqrt{2}} , \frac{1}{2 i} (x - \bar y) + \frac{w_2}{\sqrt{2}}\right ).
\end{align}
So that, ignoring constants
\begin{align}
f_j (x,y) = (\p_{\Re {w_1}} ^2 + \p_{\Re {w_2}}^2 )^j \tilde g(w_1 , w_2) |_{\substack{w_1 = \frac{1}{2} (x + \bar y) \\ w_2 = \frac{1}{2 } (x - \bar y) }}.
\end{align}
Along the diagonal, this agrees with the computation in \eqref{eq:1163} (with constants which can be shown to be equal). This provides a (non-unique) asymptotic expansion of $T_{N,f}$ using complex stationary phase, and possibly sheds light on how this method works in general.

\subsection{Composition}

In this section we let $f,g\in C^\infty_0(\C; \C)$ and determine $(f \star g)_1$. This is a simpler version of Appendix \ref{Appendix:1}, however in this case, the symplectic form is constant and in the Bergman kernel expansion, $b_j = 0$ for $j \ge 1$. First, if $f$ is a symbol, then 
\begin{align}
T_{N,f} (x,\bar y) \sim \left ( \frac{N}{ 2\pi}\right ) e^{N x \bar y} \sum _{j= 0 }^\infty N^{-j} C_j (f) (x,\bar y)
\end{align}
with
\begin{align}
C_j [f] (x,\bar y ) = \frac{1}{4^k k! } (\p_{\Re {w _1} }^2 + \p_{\Re {w_2}} ^2 ) ^j \tilde f(w_1,w_2) |_{(w_1,w_2) = \tau(x,\bar y)}
\end{align}
where $\tau(x,\bar y) = 2 ^{-1} ( x+ \bar y , i^{-1} (x-\bar y) )$. Importantly, when $y = x $ this becomes:
\begin{align}
C_j [f(\cdot ) ] (x,\bar x ) &= \frac{1}{4^j j! } (\p_{\Re {w _1} }^2 + \p_{\Re{ w_2}} ^2 ) ^j \tilde f(w_1,w_2) |_{(x,y) = \tau(x,\bar x)}= \frac{1}{j!} (\p \bar \p)^j f(x).
\end{align}
Now we may write the first few terms of $(f \star g )$:
\begin{align}
(f \star g ) _0 (x ) = C_0 [f_0 (x,\cdot ) g_0 (\cdot , \bar x) ] (x,\bar x ) =f_0 (x,\bar x ) g_0 (x,\bar x ) = f(x)g(x) ,
\end{align}
	\begin{align}	
(f \star g ) _1 (x) &= C_0[f_1 (x,\cdot )g_0 (\cdot, \bar x) ] (x,\bar x) + C_0 [f_0 (x,\cdot ) g_1 (\cdot ,\bar x) ] (x,\bar x)  \\	
&\qquad \qquad \qquad  + C_1 [f_0 (x,\cdot )g_0 (\cdot , \bar x) ] (x,\bar x) - C_1 [h_0 (\cdot ) ] (x,\bar x) \\	
&= \p\dbar f(x) g(x) + f(x) \p \dbar g(x) + C_1 [f_0 (x,\cdot ) g_0 (\cdot ,\bar x) ] (x, \bar x) - \p \dbar (f(x) g(x) )\\	
&= C_1 [f_0 (x,\cdot ) g_0 (\cdot ,\bar x) ] (x,\bar x) - \p f(x) \dbar g(x) - \dbar f (x) \p g(x)	.
\end{align}
Note:
\begin{align}
 C_1 [f_0 (x,\cdot ) g_0 (\cdot ,\bar x) ] (x,\bar x) &= \p_w \dbar _w [\tilde f (\tau (x,\bar w ) \tilde g ( w , \bar x)]|_{w = x} \\
 &= \left ( \frac{1}{2} \p_1 \tilde f + \frac{i}{2} \p_2 \tilde f \right ) \left ( \frac{1}{2} \p_1 \tilde g -\frac{i}{2} \p_2 \tilde g \right ) \\
 &= \dbar f (x) \p g(x)
\end{align}
where $\p_i \tilde f $ is the holomorphic derivative of $\tilde f$ with respect to its $i$th component. Here we use that $f_0$ is almost anti-holomorphic in the second argument and $g_0$ is almost holomorphic in the first argument. The error terms are absorbed in the $\mathcal{O}(N^{-2})$ error. We therefore get:
\begin{align}
(f \star g ) _1 (x) = -\p f(x) \dbar g(x) .
\end{align}

\smallsection{Acknowledgments} 
The author is grateful to Maciej Zworski for the countless helpful discussions and to Alix Deleporte for providing significant feedback on an earlier draft.
The author is also grateful to an anonymous referee for numerous corrections and suggestions.
This paper is based upon work jointly supported by the National Science Foundation Graduate Research Fellowship under grant DGE-1650114 and by grant DMS-1901462.

\printbibliography

\end{document}